\documentclass[a4paper,11pt,leqno,english]{amsart}
\usepackage[utf8]{inputenc}
\usepackage[T1]{fontenc}
\usepackage[french, main=english]{babel}
\usepackage{enumerate}
\usepackage{amssymb,latexsym}

\usepackage{mathrsfs,tgschola}
\usepackage[normalem]{ulem}
\usepackage{pgfplots}
\usepackage{tkz-fct}
\usepackage{url}
\usepackage{xcolor}
\usepackage{comment}
\definecolor{violet}{rgb}{0.0,0.2,0.7}
\definecolor{rouge2}{rgb}{0.8,0.0,0.2}
\usepackage{tikz}
\usepackage{empheq}
\usepackage{tikz-cd}
\usetikzlibrary{matrix,arrows,decorations.pathmorphing}
\usepackage{hyperref}
\usepackage{mathpazo}
\usepackage{enumerate}

\hypersetup{
    bookmarks=true,         
    unicode=false,          
    pdftoolbar=true,        
    pdfmenubar=true,        
    pdffitwindow=false,     
    pdfstartview={FitH},    
    pdftitle={},    
    pdfauthor={},     
    colorlinks=true,       
   linkcolor=rouge2,          
    citecolor=violet,        
    filecolor=black,      
    urlcolor=cyan}           
\usepackage{enumitem}
\usepackage{appendix}

 \theoremstyle{plain}    
 \newtheorem{thm}{Theorem}[section]
\theoremstyle{plain} 
\theoremstyle{thm*} 

 \numberwithin{equation}{section} 
 \numberwithin{figure}{section} 
    
 \newtheorem{cor}[thm]{Corollary} 
 \theoremstyle{plain}    
 \newtheorem{prop}[thm]{Proposition} 
 \theoremstyle{plain}    
  \theoremstyle{thm*}    
  \newtheorem*{question*}{Question}
 \newtheorem{lem}[thm]{Lemma} 
 \theoremstyle{remark}
 \theoremstyle{remark}
 \newtheorem{rem}[thm]{Remark}
 \theoremstyle{definition}
\newtheorem{exa}[thm]{Example}
\theoremstyle{plain}  
\newtheorem{set}[thm]{Setup}
\theoremstyle{plain}

\theoremstyle{definition}
\newtheorem{defi}[thm]{Definition}
\newtheorem{notation}[thm]{Notation}

\newtheorem{conj}[thm]{Conjecture}

\newcommand{\C}{{\mathbb{C}}}
\newcommand{\N}{{\mathbb{N}}}
\newcommand{\PP}{{\mathbb{P}}}
\newcommand{\Q}{{\mathbb{Q}}}
\newcommand{\R}{{\mathbb{R}}}

\newcommand{\Z}{{\mathbb{Z}}}

\newcommand{\cC}{{\mathcal{C}}}
\newcommand{\cD}{{\mathcal{D}}}

\newcommand{\cF}{{\mathcal{F}}}
\newcommand{\cG}{{\mathcal{G}}}

\newcommand{\cO}{{\mathcal{O}}}

\newcommand{\cX}{{\mathcal{X}}}
\newcommand{\cY}{{\mathcal{Y}}}
\newcommand{\cZ}{{\mathcal{Z}}}

\def\1{\bold{1}}

\newcommand{\bD}{{\mathbb{D}}}

\newcommand{\om}{\omega}
\newcommand{\Xr}{X_{\rm reg}}
\newcommand{\Xs}{X_{\rm sing}}
\newcommand{\wX}{\widetilde X}

\newcommand{\Sk}{S^{[k]}}
\newcommand{\SkF}{S^{[k]}\cF}

\newcommand{\SiF}{S^{[i]}\cF}

\DeclareMathOperator{\Ric}{Ric}

\renewcommand{\ge}{\geqslant}
\renewcommand{\le}{\leqslant}
\newcommand{\wt}{\widetilde}

\newcommand{\Aut}{\operatorname{Aut}}
\newcommand{\Alb}{\operatorname{Alb}}

\newcommand{\GL}{\operatorname{GL}}
\newcommand{\U}{\operatorname{U}}
\newcommand{\SU}{\operatorname{SU}}
\newcommand{\Sp}{\operatorname{Sp}}
\newcommand{\GLn}{\operatorname{GL}(n,\mathbb C)}
\newcommand{\Un}{\operatorname{U}(n)}

\newcommand{\PSH}{\operatorname{PSH}}
\newcommand{\PH}{\operatorname{PH}}

\newcommand{\codim}{\operatorname{codim}}
\newcommand{\Hol}{\operatorname{Hol}}
\newcommand{\Holo}{\operatorname{Hol}^\circ}
\newcommand{\Dlt}{\operatorname{Def}^{\rm lt}}

\title{Beauville-Bogomolov decomposition for klt varieties}

\author{Henri Guenancia}
\address{Univ. Bordeaux, CNRS, Bordeaux INP, IMB, UMR 5251, F-33400 Talence, France}\email{henri.guenancia@math.cnrs.fr}

\date{}

\begin{document}

\begin{abstract}
These lecture notes present a mostly self-contained proof of the singular version of Beauville-Bogomolov decomposition theorem for compact Kähler varieties with log terminal singularities and zero first Chern class.
\end{abstract}

\maketitle
\tableofcontents

\section{Introduction}

\subsection{The decomposition theorem}
Let $X$ be a compact Kähler manifold such that $c_1(X)=0\in H^2(X,\C)$. The celebrated Beauville-Bogomolov decomposition theorem \cite{Bea83,Bog74} states that up to a finite étale cover, $X$ splits as a product of a complex torus and simply connected manifolds of a very particular type. 

\begin{thm}[Beauville-Bogomolov decomposition theorem]
Let $X$ be a compact Kähler manifold such that $c_1(X)=0\in H^2(X,\C)$. There exists a finite étale cover $f:X'\to X$ and a product decomposition
\[X'=T\times \prod_{i\in I} Y_i \times \prod_{j\in J} Z_j\]
where $T$ is a complex torus and for any index $i\in I$ (resp. $j\in J$), the manifold $Y_i$ (resp. $Z_j$) is irreducible Calabi--Yau (resp. irreducible holomorphic symplectic). 
\end{thm}

The terminology of irreducible Calabi--Yau (ICY for short) and irreducible holomorphic symplectic (IHS for short) is not necessarily standard in the literature. Here, we say that a compact Kähler manifold $X$ of dimension $n$ such that $K_X\sim \cO_X$ and $\pi_1(X)=\{1\}$ is
\begin{enumerate}[label=$\bullet$]
\item \emph{irreducible Calabi--Yau} if $n\ge 3$ and for any integer $0<p<n$, one has $H^0(X,\Omega_X^p)=\{0\}$. 
\item \emph{irreducible holomorphic symplectic} if there exists a holomorphic symplectic form $\sigma \in H^0(X,\Omega_X^2)$ such that we have an isomorphism of algebras $\oplus_{p=0}^n H^0(X,\Omega_X^p)=\C[\sigma]$. 
\end{enumerate}

\subsection{Skeleton of the proof}

The starting point of the proof is the resolution of the Calabi conjecture by Yau \cite{Yau78} which states that for any compact Kähler manifold $X$ with $c_1(X)=0$, there exists a Kähler metric $\omega$ such that $\Ric \om=0$. 

As an immediate corollary of Bochner formula, any global holomorphic tensor $\tau \in H^0(X,T_X^{\otimes p}\otimes \Omega_X^{\otimes q})$ is parallel with respect to $\omega$. So, given $x\in X$, there is a one-to-one correspondence between global holomorphic tensors of a given type $(p,q)$ and vectors in $T_{X,x}^{\otimes p}\otimes \Omega_{X,x}^{\otimes q}$ that are invariant under the holonomy group of $(X,\omega)$. This is often referred to as the Bochner principle. 

Then, the de Rham decomposition theorem shows that the universal cover $\pi:\wX\to X$ splits isometrically as $(\wX,\pi^*\omega)=(\C^k, \omega_{\C^k})\times \prod_{i\in I} (X_i,\omega_i)$ where $(X_i,\om_i)$ is a simply connected, Kähler Ricci flat manifold  whose holonomy $G_i$ is irreducible. 

Thanks to Berger-Simons classification for holonomy groups, one must have $G_i=\SU(\C^{n_i})$ or $G_i=\Sp(\C^{n_i})$ where $n_i=\dim_{\C}X_i$.

Next, Cheeger-Gromoll splitting theorem \cite{CG71} implies that the manifolds $X_i$ are compact. Therefore, Bochner principle applies to them and the a priori knowledge of their holonomy shows that they are either ICY or IHS manifolds, depending on their holonomy type. 

Then, one can show (cf \cite{Bea83}) that there exists a finite index subgroup $H$ of $\pi_1(X)$ which preserves the product decomposition of $\wX$. The final input is to appeal to Bieberbach theorem which guarantees that there is a finite index subgroup of $H$ whose action on $\mathbb C^{k}$ is by a lattice of translations. Putting everything together yields the decomposition theorem. 

\subsection{The singular version}
When attempting to classify smooth projective varieties or compact Kähler manifolds up to birational or bimeromorphic equivalence, a key object is the Iitaka fibration. It provides for any such variety (say not of general type) a bimeromorphic model which is the total space of a nontrivial algebraic fiber space whose fibers have vanishing Kodaira dimension, highlighting the central role of the latter varieties. 

One of the main conjectures from the analytic version of the Minimal Model Program, combining the existence of a minimal model and the abundance conjecture, can be stated as follows

\begin{conj}
Let $X$ be a compact Kähler manifold with $\kappa(X)=0$. There exists a bimeromorphic map
\[X\dashrightarrow X_{\rm min}\]
to a normal compact Kähler variety $X_{\rm min}$ with terminal singularities such that $K_{X_{\rm min}}$ is torsion. 
\end{conj}  

The conjecture above is still widely open even for projective manifolds. However, it provides a legitimate motivation to try to understand the structure of terminal varieties with torsion canonical bundle in the spirit of the decomposition theorem above, and one of the first attempts date back to Peternell \cite{Pet94} in the early 90s. This has been achieved (in a more general setting) thanks to the recent series of articles settling the projective case first \cite{Dru16,GGK,HP} and then the Kähler case \cite{BGL}.

\begin{thm}
\label{thm BB klt}
Let $X$ be a normal compact Kähler variety with klt singularities satisfying $c_1(X)=0\in H^2(X,\C)$. There exists a finite quasi-étale cover $f:X'\to X$ and a product decomposition
\[X'=T\times \prod_{i\in I} Y_i \times \prod_{j\in J} Z_j\]
where $T$ is a complex torus and for any index $i\in I$ (resp. $j\in J$), the variety $Y_i$ (resp. $Z_j$) is irreducible Calabi--Yau (resp. irreducible holomorphic symplectic). 
\end{thm}

We refer to the next section for the meaning of the various notions and objects involved in the statement above, and will only comment on the following two things:
 
$\bullet$  Part of the difficulty in this singular setting is to come up with a suitable notion of irreducible Calabi--Yau variety and irreducible holomorphic symplectic variety. Many different definitions had been proposed in the past but the one appearing in the theorem above was introduced by Greb-Kebekus-Peternell \cite{GKP}. We refer to Section~\ref{sec irr} for a brief overview and to \cite[Section~14]{GGK} for an in-depth comparison of existing definitions along with many examples.

$\bullet$ The proof of the Kähler case in \cite{BGL} heavily relies on the projective case by a deformation theoretic argument. This is a major difference with what happens in the smooth case. 

\subsection{Brief outline of the proof}

We indicate below which elements of the proof in the smooth case go through in the singular case (possibly with additional work required), and which ones outright collapse. Then, we point out the major new inputs to complete the proof.\\

{\it What survives. }

\noindent
$-$ The existence of a Ricci flat Kähler metric $\omega$ on $\Xr$, cf Theorem~\ref{thm EGZ}.

\noindent
$-$ The Bochner principle on $(\Xr,\omega)$, cf Theorem~\ref{bochner}.

\noindent
$-$ The identification of the {\it restricted} holonomy of $(\Xr, \omega)$, cf Section~\ref{sec rest hol}. \\

{\it What doesn't. }

\noindent
Essentially, things crash when trying to appeal to de Rham splitting theorem for the universal cover of $\Xr$, due to the incompleteness of $(\Xr,\omega)$. Even if one could prove such a result, it is unclear what the analog of Cheeger-Gromoll's theorem should be in this setting (maybe a suitable compactification theorem for the non-euclidean factors in the universal cover).\\

{\it New inputs. }

\noindent
$-$ The identification of flat directions in the tangent with $1$-forms on quasi-étale covers, cf Theorem~\ref{thm flat}.

\noindent
$-$ Finiteness of the number of connected components of the holonomy of $(\Xr, \omega)$, cf Theorem~\ref{thm holonomy}.

\noindent
$-$ Algebraic integrability of the non-flat factors in the projective case, cf Theorem~\ref{thm HP}.

\noindent
$-$ Splitting theorem for algebraically integrable foliations, cf Theorem~\ref{split druel}. 

\bigskip

{\bf Acknowledgements.} These lecture notes are based on a series of lectures given at the C.I.M.E summer school on Calabi-Yau varieties in 2024. I am very grateful to the organizers Simone Diverio, Vincent Guedj, and Hoang Chinh Lu for the kind invitation to participate. I am indebted to Hyunsuk Kim and the anonymous referees for their very thorough and careful reading of these notes and for their numerous valuable remarks and suggestions which helped improve the exposition. Finally, I would like to thank my collaborators Ben Bakker, Junyan Cao, Benoît Claudon, Stéphane Druel, Patrick Graf, Daniel Greb, Stefan Kebekus, Christian Lehn and Mihai P\u{a}un for all the fruitful discussions we have had over the years on the topics appearing in these notes.

\section{Preliminary material}

Before getting on to the material let us give a few additional references for introductory reading on the several topics that will be touched on here, like Kähler holonomy \cite{Besse, Joy00}, holomorphic foliations \cite{Brunella}, currents on singular spaces \cite{Dem85,BEG}, singularities and the Minimal Model Program \cite{Kollar97, KollarMMP}.

\subsection{Varieties}
In what follows, a variety will be an irreducible and reduced complex space.  

\begin{defi}[Quasi-étale covers]
Let $f:Y\to X$ be a morphism between normal varieties. One says that $f$ is a cover if $f$ is finite and surjective. Moreover, one says that $f$ is quasi-étale if $f$ is étale in codimension one. That is, there exists $Z\subset Y$ proper analytic subset of codimension at least two such that $f|_{Y\setminus Z}$ is étale onto its image. 
\end{defi}

We will repeatedly use the following two crucial properties of quasi-étale covers. First, Nagata's purity of branch locus asserts that if $f:Y\to X$ is a cover, then $f$ is quasi-étale if, and only if that $f$ is étale over $\Xr$. Next, let $X$ be a normal variety and let $f^\circ:Y^\circ\to \Xr$ be an étale cover. Then,  $f^\circ$ admits a unique extension to a quasi-étale cover $f:Y\to X$ where $Y$ is a normal variety containing $Y^\circ$ as a dense open subset.

\begin{defi}[Kähler variety]
Let $X$ be a normal variety. A Kähler form on $X$ is a Kähler form $\omega_X$ on $X_{\rm reg}$ such that given any local embedding $X\underset{\rm loc}{\hookrightarrow} \C^N$, $\omega_X$ extends to a Kähler form on $\C^N$. One says that $X$ is Kähler if it admits a Kähler form.  
\end{defi}

If $\omega_X$ is a Kähler form, then there exists a covering $X=\cup U_\alpha$ by Stein subsets such that $\omega_X$ is given by a collection of smooth strictly plurisubharmonic (psh) functions $\phi_\alpha$ on $U_\alpha$ such that $\phi_\alpha-\phi_\beta$ is pluriharmonic on $U_{\alpha\beta}=U_\alpha \cap U_\beta$. In particular, a Kähler form determines an element in $H^0(X,\cC^\infty/\mathrm{PH}_X)$ where $\mathrm{PH}_X$ is the sheaf of pluriharmonic functions. Denote by $\{\omega_X\}$ the image of $\omega_X$ under the connecting map \[H^0(X,\cC^\infty/\mathrm{PH}_X) \longrightarrow H^1(X, \mathrm{PH}_X).\] We have another connecting homomorphism $\delta:H^1(X, \mathrm{PH}_X)\to H^2(X, \C)$ coming from the exact sequence \[0\longrightarrow \R\overset{i\cdot}{\longrightarrow} \cO_X\overset{\mathrm{Re}}{\longrightarrow} \mathrm{PH}_X\longrightarrow 0.\]  We set $[\omega_X]=\delta(\{\omega_X\})\in H^2(X,\R)$.  Such a class is called a Kähler class.

\subsection{Sheaves}
\begin{defi}[Reflexive differentials]
Let $X$ be a normal variety of dimension $n$ and let $\Omega_X$ be the sheaf of Kähler differentials. For any integer $p\in [0,n]$, we define the sheaf of reflexive differentials to be the reflexive hull $\Omega_X^{[p]}:=(\Lambda^p \Omega_X)^{**}$. The tangent sheaf is the dual  $T_X:=\Omega_X^*$; it is a reflexive coherent sheaf. 
\end{defi}

\begin{notation}[Reflexive symmetric powers and pullback]
Let $X$ be a normal variety, let $\cF$ be a reflexive sheaf on $X$. If $k\in \N^*$ is an integer, we denote by $S^{[k]}\cF:=(S^k\cF)^{**}$ the reflexive hull of the $k$-th symmetric tensor power of $\cF$.  If $Y$ is a normal variety admitting a morphism $f:Y\to X$, we denote by $f^{[*]}\cF:=(f^*\cF)^{**}$ the reflexive hull of the pullback of $\cF$ by $f$.  
\end{notation}

\begin{defi}[Foliations]
Let $X$ be a normal variety. A foliation on $X$ is a coherent subsheaf $\cF\subset T_X$ such that $\cF$ is saturated (that is, $T_X/\cF$ is torsion-free) and $\cF|_{\Xr}$ is stable under Lie bracket. Let $X^\circ$ be the largest open subset of $X_{\rm reg}$  on which $\cF$ is a subbundle of $T_X$.

 A \emph{leaf} of $\cF$ is a connected, locally closed holomorphic submanifold $L\subset X^\circ$ such that $T_L=\cF|_L$. A leaf is called analytic (or algebraic if $X$ is algebraic) if it is open in its Zariski closure, or in other words if $\dim \overline L^{\rm Zar} = \mathrm{rk}(\cF)$. 
 
 The foliation $\cF$ is called analytically integrable (or algebraically integrable if $X$ is algebraic) if its leaves are analytic (resp. algebraic). 
\end{defi}

\begin{defi}[Stable sheaves]
Let $\cF$ be a coherent reflexive sheaf of rank $r$ on a normal compact Kähler variety $X$ of dimension $n$ and let $\alpha \in H^2(X, \C)$ be a Kähler class. One says that $\cF$ is stable with respect to $\alpha$ if for any subsheaf $\cG\subset \cF$ or rank $0<s<r$, we have
\[\mu_{\alpha}(\cG)=\frac 1 s c_1(\cG) \cdot \alpha^{n-1} < \frac 1 r c_1(\cF) \cdot \alpha^{n-1} =\mu_{\alpha}(\cF). \]
The number $c_1(\cF) \cdot \alpha^{n-1} =\mu_{\alpha}(\cF)$ can be defined as $c_1(f^*\cF)\cdot (f^*\alpha)^{n-1}$ where $f:\wX\to X$ is any resolution of singularities. 

One says that $\cF$ is polystable with respect to $\alpha$ if one can write $\cF= \oplus_{i=1}^k \cF_i$ where each $\cF_i$ is $\alpha$-stable of slope $\mu_\alpha(\cF)$. 

One say that $\cF$ is strongly stable with respect to $\alpha$ if for any $f:Y\to X$ quasi-étale cover, the reflexive pull-back $f^{[*]}\cF=(f^*\cF)^{**}$ is stable with respect to $f^*\alpha$. 
\end{defi}

\begin{exa}
Let $S$ be a K3 surface and let $X:=(S\times S)/\langle i \rangle$ where $i(s,t)=(t,s)$ is the natural involution. Then $T_X$ is stable with respect to any Kähler class but $T_{S\times S}$ is not, while the degree two cover $S\times S\to X$ is quasi-étale. Thus, $T_X$ is not strongly stable.   
\end{exa}

\subsection{Holonomy}

Let $(E,\nabla)$ be a vector bundle or rank $r$ equipped with a connection on a differentiable manifold $X$ and let $x\in X$. Given any loop $\gamma $ at $x$, parallel transport along $\gamma$ yields an element $\tau_\gamma \in \mathrm{GL}(E_x)$. We define  $\Hol_x(E,\nabla):=\{\tau_\gamma, \, \gamma \, \mbox{loop at } \, x\}$. Similarly, one defines $\Holo_x(E,\nabla)$ using only loops which are homotopic to the constant loop. The subgroup $\Holo_x(E,\nabla)\subset \Hol_x(E,\nabla)$ is normal, it is a connected Lie subgroup of $\GL(E_x)$ and there is a surjection 
\begin{equation}
\label{pi 1 hol}
\pi_1(X)\twoheadrightarrow \Hol_x(E,\nabla)/\Holo_x(E,\nabla).
\end{equation}
If one views $\Hol_x(E,\nabla)\subset \mathrm{GL}(\R^r)$, then changing the basepoint $x$ yields a conjugate subgroup so that the conjugacy class of $\Hol_x(E,\nabla)$ inside $\mathrm{GL}(\R^r)$ is well-defined independently of $x$. 

If the connection is flat (that is, $\nabla^2=0$), then parallel transport only depends on the homotopy class of the loop, hence $\Holo(E,\nabla)=\{1\}$. Conversely, if $\Holo(E,\nabla)=\{1\}$, then one can unambiguously use parallel transport to extend a base of $E_x$ to a parallel frame of $E$ in a simply connected neighborhood $U$ of $x$. In particular, $\nabla^2=0$ on $U$. In summary, one has
\begin{equation}
\label{hol flat}
(E,\nabla) \,\, \mbox{is flat } \quad \Longleftrightarrow \quad \Holo(E,\nabla)=\{1\}. 
\end{equation}

\begin{defi}[Kähler holonomy]
If $(X,\omega)$ is a Kähler manifold of complex dimension $n$ and $x\in X$, then we define $\Hol_x(X,\omega)$ to be the holonomy of the holomorphic hermitian bundle $(T_X, \nabla^h)$ where $\nabla^h$ is the Chern connection associated to the hermitian metric $h$ on $T_X$ induced by $\omega$. Since both the complex structure and the hermitian metric are parallel, one can view $\Hol_x(X,\omega)\subset \Un$. 
\end{defi}

Note that if $f:Y\to X$ is an étale cover and $y\in f^{-1}(x)$, then every loop at $x$ which is homotopic to a constant loop can be lifted to a loop at $y$ hence there is a canonical isomorphism $\Holo_x(X,\omega)\simeq \Holo_y(Y,f^*\omega)$.

\section{Klt singularities: topological, analytic and metric aspects}

\subsection{Definition and examples}
Let $X$ be a normal variety of dimension $n$ such that $K_X:=(\Lambda^n \Omega_X)^{**}$ is a $\Q$-line bundle. That it, there exists an integer $m\ge 1$ such that the reflexive hull $K_X^{[m]}:=(K_X^{\otimes m})^{**}$ is locally free, i.e. it is a line bundle. From now one, we will use the additive notation for powers of line bundles. Let $\pi:\wX\to X$ be a log resolution of singularities of $X$ with exceptional locus a divisor $E=\sum_{i=1}^r E_i$ with simple normal crossing support. One can canonically write
\[K_{\wX/X}:=K_{\wt X}-\pi^*K_X= \sum_{i=1}^r a_i E_i\]
where $a_i \in \mathbb Q$ is called the discrepancy of $E_i$. 

\begin{defi}
Let $X$ be a normal variety and let $\pi : \wX\to X$ be a resolution as above. One says that $X$ has Kawamata log terminal (klt for short) singularities (resp. canonical singularities, terminal singularities) if for any log resolution $\pi : \wX\to X$ as above, one has $a_i>-1$ (resp. $a_i \ge 0$, $a_i>0$) for every index $i$. 
\end{defi}

\begin{exa}
Any finite quotient singularity (i.e. locally analytically isomorphic to $(\mathbb C^n/G, 0)$ where $G\subset \GLn$ is finite) is klt. 
\end{exa}

\begin{exa}
The $A_1$ singularity $Q:=(z_0^2+\cdots +z_n^2=0)\subset \C^{n+1}$ is klt for any $n\ge 2$. However, if $n\ge 3$, this is not a quotient singularity. There are several ways to see it. 

The most direct one is to observe that the local fundamental group of the singularity is trivial; that is $\pi_1(Q\setminus \{0\})=\{1\}$. Let us sketch a proof of that fact. Denote by $\overline Q$ the smooth projective quadric of dimension $n-1$. We have a $\C^*$-fiber bundle $\pi: Q\setminus \{0\}\to \overline Q$ which is nothing but the restriction of $L:=\cO_{\overline Q}(-1)\to \overline Q$ to the complement of the zero section. The fibration $\pi$ induces en exact sequence 
\[\pi_2(\overline Q)\to \Z=\pi_1(\C^*)\to \pi_1(Q\setminus \{0\})\to \pi_1(\overline Q)\to \{1\}.\]
 Since $\overline Q$ is simply connected, all we have to do is prove that the map $\partial:\pi_2(\overline Q)\to \pi_1(\C^*)\simeq \Z$ is onto. It is a classical fact (see e.g. \cite[p70-]{BottTu}) that $\partial$ corresponds to $c_1(L)\in H^2(X,\Z)$ under the isomorphism $\pi_2(\overline Q)\to H_2(X,\Z)$ provided by Hurewicz's theorem.  Now Lefschetz hyperplane theorem implies that the natural map $H_2(\overline Q,\Z)\to H_2(\PP^n, \Z)$ is surjective since $n\ge 3$. Since pairing with $c_1(\cO_{\PP^n}(-1))$ induces an isomorphism $H_2(\PP^n, \Z)\to \Z$, we deduce that $\partial$ is surjective, hence the result. 

Alternatively, one can invoke \cite{Sch71} stating that isolated, quotient singularities of dimension at least three are rigid but $Q$ is obviously smoothable.  
\end{exa}

\subsection{Topological properties of klt singularities}

As we have seen above, klt singularities are not finite quotient singularities in general. Yet, one of the paramount achievements in understanding the local topology of klt singularities is the following theorem showing that just like quotient singularities, their regional fundamental group is finite. 

\begin{thm}[\cite{Braun21}]
\label{thm braun}
Let $X$ be a normal variety with klt singularities, and let $x\in X$. There exists a (small) euclidean neighborhood $U$ of $x$ such that $\pi_1(U_{\rm reg})$ is finite. 
\end{thm}

This fundamental result builds on earlier work by Tian-Xu \cite{TianXu17} and Greb-Kebekus-Peternell \cite{GKP16}. Note that the above results actually deal with algebraic varieties, but they extend to analytic varieties thanks to Fujino's results \cite{Fujino22} on the relative MMP for projective morphisms between analytic spaces as explained e.g. in \cite[Remark~6.10]{CGGN}.\\

Let us now mention a couple of striking applications of Theorem~\ref{thm braun}. In order to do so, we need to introduce the following notion, originating in \cite{GKP16}. 

\begin{defi}
\label{mqe}
Let $X$ be a normal variety. A maximally quasi-étale cover of $X$ is a quasi-étale Galois cover $f:Y\to X$ such that the natural map between the étale fundamental groups
\[\hat\pi_1(Y_{\rm reg})\longrightarrow \hat \pi_1(Y)\]
is an isomorphism. 
\end{defi}

If $f:Y\to X$ is maximally quasi-étale, then it follows from Mal\v{c}ev theorem that any finite dimension linear representation of $\pi_1(Y_{\rm reg})$ factors through $\pi_1(Y)$. Equivalently, any flat bundle over $Y_{\rm reg}$ extends to a flat bundle on $Y$. 

We will state the next result for compact varieties for simplicity, but the result is local in nature. 

\begin{cor}
\label{mqe}
Let $X$ be compact normal variety with klt singularities. Then $X$ admits a maximally quasi-étale cover. 
\end{cor}

The result was proved in \cite[Theorem~1.5]{GKP16} for quasi-projective varieties relying on the finiteness of the fundamental group of the link of a klt singularity proved by \cite{TianXu17}. Here again, everything extends to the analytic setting thanks to \cite{Fujino22}. 

\begin{cor}
\label{torus quotient}
Let $X$ be compact normal Kähler variety with klt singularities such that $T_{\Xr}$ is flat. Then there exists a quasi-étale Galois cover $T\to X$ where $T$ is complex torus. 
\end{cor}

The result was proved for $X$ projective in \cite[Corollary~1.16]{GKP16} and in \cite[Theorem~D]{CGGN}. The proof is an easy consequence of Corollary~\ref{mqe}. Indeed, the latter enables us to replace $X$ with a maximally quasi-étale cover $Y$. Then $T_Y$ will be flat, hence locally free, but the solution of Zariski-Lipman conjecture for klt singularities \cite{DruelZL,GKK} shows that $Y$ is smooth and we are reduced to the well-known smooth case (building upon Yau's solution of the Calabi conjecture). 

Alternatively, one could work locally using Theorem~\ref{thm braun} to show that $X$ has only quotient singularities, and then use the orbifold version of the uniformisation theorem for flat orbifolds. More precisely, once we know that $X$ is an orbifold, it is elementary to check that $T_X$ is a flat orbi-bundle. Then the Calabi-Yau metric $\omega$ is an orbifold metric and since the orbifold Chern classes of $X$ vanish, one sees as in the smooth case that $\omega$ is flat. The orbifold universal cover of $(X,\omega)$ is a complete flat Kähler {\it manifold} thanks to \cite[Corollary~2.16 on p. 603]{BH99}, hence it is biholomorphically isometric to $(\mathbb C^n, \omega_{\rm eucl})$. The result then follows from Bieberbach's theorem.

\subsection{Analytic properties of klt singularities}

It is well-known that klt singularities are rational. That is, given any resolution of singularities $\pi:\wX\to X$, we have $R^j \pi_*\cO_{\wX}=0$ for $j\ge 1$. In this short paragraph, we recall the following extension theorem for reflexive differential form on varieties with normal singularities. 

\begin{thm}[\cite{GKKP11, KS}]
\label{thm extension forms}
Let $X$ be a normal variety with rational singularities and let $\pi:\wX\to X$ be a resolution of singularities. Then for any integer $p\in [0,n]$, the sheaf $\pi_*\Omega_{\wX}^p$ is reflexive, hence coincides with $\Omega_X^{[p]}$. 
\end{thm}

In other words, if $\sigma \in H^{0}(\Xr, \Omega_{\Xr}^p)$, then $\pi^*\sigma$ extends to a holomorphic $p$ form across $\mathrm{Exc}(\pi)$. 

\begin{cor}
\label{hodge duality}
Let $X$ be a compact normal variety with rational singularities. Then for any integer $p\in [0,n]$, one has $h^p(X,\cO_X)=h^0(X, \Omega_X^{[p]})$. 
\end{cor}

Indeed, one has $h^p(X,\cO_X)=h^p(\wX,\cO_{\wX})$ and by Hodge theory the latter coincides with $h^0(\wX, \Omega_{\wX}^p)$ and the corollary follows from the theorem above. 

\subsection{Kähler-Einstein metrics on klt varieties}

Let $X$ be a normal variety with klt singularities such that $c_1(X)=0\in H^2(X, \R)$. Here, $c_1(X)$ can be viewed as $-\frac 1m c_1(mK_X)\in H^2(X, \Q)$ if $m$ is divisible enough so that $mK_X$ is a line bundle so that one can consider its image under the usual map $H^1(X, \cO_X^*)\to H^2(X,\Z)$. 

It makes sense to consider a smooth hermitian metric $h_0$ on $K_X$. Essentially by definition, this is the datum of a (suitable) covering $X=\cup U_\alpha$ by Stein spaces where $mK_{U_{\alpha}}$ is trivialized by a holomorphic section $\sigma_\alpha$ as well as smooth functions $\phi_\alpha$ on $U_\alpha$ (i.e. $|\sigma_\alpha|_{h_0^{\otimes m}}=e^{-m\phi_\alpha}$) such that the cocycle $\big(m(\phi_\alpha-\phi_\beta)|_{U_{\alpha\beta}}\big)_{\alpha, \beta}\in H^1(X,\mathrm{PH}_X)$ is mapped to $c_1(mK_X)$ in $H^2(X,\C)$. 

Since the latter vanishes, one can find pluriharmonic functions $\psi_\alpha$ on $U_\alpha$ such that $\phi_\alpha-\phi_\beta=\psi_\alpha-\psi_\beta$ on $U_{\alpha \beta}$. In other words, the functions $\phi_\alpha-\psi_\alpha$ glue to a global function $f$ on $X$ and the local weights $e^{-\psi_\alpha}$ define a smooth hermitian metric $h=e^fh_0$ on $K_X$ such that its Chern curvature $\Theta(K_X, h)$ vanishes. 

The local volume forms $\frac{(\sigma_\alpha \wedge \overline \sigma_\alpha)^{\frac 1m}}{|\sigma_\alpha|_{h^{\otimes m}}^{\frac 2m}}= e^{-\psi_\alpha} (\sigma_\alpha \wedge \overline \sigma_\alpha)^{\frac 1m}$ on $U_{\alpha}^{\rm reg}$ have finite mass thanks to the klt condition (see \cite[Lemma~6.4]{EGZ}), and they glue to define a smooth volume form on $\Xr$, independent of $m$ and the choice of the trivializations $\sigma_\alpha$. We extend the associated measure trivially across $\Xs$ to obtain a Radon measure $\mu_h$.

Now, assume that $(X,\omega_X)$ is a compact Kähler variety of dimension $n$ with klt singularities such that $c_1(X)=0$. Let $h$ be a smooth hermitian metric on $K_X$ with zero Chern curvature and let $\mu_h$ the associated measure, normalized so that $\int_X d\mu_h=\int_X\omega_X^n$. 

\begin{thm}[\cite{EGZ, Paun}]
\label{thm EGZ}
In the setup above, there is a unique function $\varphi \in \PSH(X,\omega_X)\cap L^{\infty}(X)$ with $\sup_X \varphi =0$ such that 
\[(\omega_X+dd^c \varphi)^n = d\mu_h.\]
Moreover, the restriction of  $\omega:=\omega_X+dd^c \varphi$ to $\Xr$ is a genuine Ricci flat Kähler metric on $\Xr$. 
\end{thm}

We usually refer to $\omega$ as the singular Ricci flat Kähler metric in $[\omega_X]$. It is not difficult to see that if $f:Y\to X$ is a quasi-étale cover, then $f^*\omega$ is the singular Ricci flat Kähler metric in the Kähler class $f^*[\omega_X]$. 

Properties of $\omega$ near $\Xs$ are still vastly mysterious, aside from orbifold singularities \cite{LiTian19,GP24} and $A_1$ singularities \cite{HS}. 

On the negative side, it is an elementary application of a theorem of Yau that the Kähler manifold $(\Xr, \omega)$ is never geodesically complete unless $X$ is smooth, cf \cite[Proposition~4.2]{GGK}. Actually, one can improve the latter statement by saying that the diameter of $(\Xr, \omega)$ is finite \cite{GPSS23}. 

\section{Polystability, Bochner principle and Albanese map}

It is convenient to set our notation for what follows. 

\begin{set}
\label{setup}
Let $X$ be a compact Kähler variety of dimension $n\ge 2$ with klt singularities such that $c_1(X)=0\in H^2(X,\C)$. Let  $\alpha\in H^2(X,\R)$ be a Kähler class and let $\omega \in \alpha$ be the singular Kähler Ricci flat metric constructed in Theorem~\ref{thm EGZ}. We denote by $\nabla$ the Chern connection on $T_{\Xr}$ with respect to the hermitian metric $h$ induced by $\omega|_{\Xr}$. Finally, we fix a basepoint $x\in \Xr$ and set $G:=\Hol_x(\Xr,\omega)$ which acts on $V:=T_{X,x}$ by unitary transformations wrt $h_x$.
\end{set}

\subsection{The canonical decomposition of $T_X$}

\label{sec can dec}
The starting point of this whole story can reasonably be traced back to the following infinitesimal version of the decomposition theorem, proved independently by \cite{GKP, GSS}. 

\begin{thm}[Canonical decomposition of the tangent sheaf]
\label{thm poly}
In Setup~\ref{setup}, there exists a decomposition of the tangent sheaf 
\[T_X=\bigoplus_{i\in I} \cF_i\]
where the the reflexive sheaves $\cF_i$ satisfy the following properties
\begin{enumerate}[label=$(\roman*)$]
\item $\cF_i$ is stable with respect to $\alpha$ and $c_1(\cF_i)=0\in H^2(X,\R)$. 
\item $\cF_i$ is a foliation, i.e. ${\cF_i}|_{\Xr}$ is stable under Lie bracket. 
\item ${\cF_i}|_{\Xr}$ is parallel with respect to $\nabla$, that is $\nabla \cF_i \subset \cF_i$ over $\Xr$.
\item For any $i\neq j$, ${\cF_i}|_{\Xr}$ and ${\cF_j}|_{\Xr}$ are orthogonal with respect to $h$. 
\end{enumerate}
Moreover, the refined decomposition $T_X= \cO_X^{\oplus k} \oplus \bigoplus_{i\in I'} \cF_i$ none of the $\cF_i$ are trivial for $i\in I'$ is unique up to permutation of the factors. 
\end{thm} 

\begin{rem}
\label{rem strong}
From this enhanced polystability property of $T_X$ with respect to an arbitrary Kähler class $\alpha$, it follows that the sheaves $\cF_i$ above are actually stable with respect to {\it any} Kähler class, cf e.g. the second paragraph of \cite[p270]{CGGN}.
\end{rem}

The statement in \cite{GKP} actually contains an additional property (that the sheaves $\cF_i$ can be assumed to be strongly stable up to taking a quasi-étale cover). However, it holds only for projective varieties with canonical singularities and does not cover the canonical aspect of the decomposition nor the metric features of the above result.  Semistability of $T_X$ has been known for a long time: this is a consequence of Miyaoka's generic semipositivity theorem when $X$ is projective with canonical singularities and was proved by \cite{Enoki} in the Kähler setting. However, the crucial polystability property was out of reach at that time; in the analytic setting, this has to do with the fact that the existence of $\omega$ was unknown then. \\

\begin{proof}[A few ideas about the proof]
We will not give the details of the proof of Theorem~\ref{thm poly}, but only the rough strategy, glossing over any technicality. 

\emph{Smooth case.} Assume for a moment that $X$ is smooth. The (classical) idea is to notice that $h$ is a Hermite-Einstein metric on $T_X$ with respect to $\omega$ and then appeal to the easy direction of the Kobayashi-Hitchin correspondence.  More precisely, a standard consequence of Griffiths inequality is that if one endows a subbundle $F\subset T_X$ with the induced metric, then one has pointwise 
\[ c_1(F,h|_F) \wedge \omega^{n-1}  \le  \mathrm{pr}_F (c_1(T_X,h)|_F) \wedge \omega^{n-1}\]
and the inequality is strict at a given point iff the second fundamental form of $(F,h_F)\subset (T_X, h)$ is not zero. 
By the Hermite-Einstein condition and since $c_1(T_X)=0$, the RHS identically zero. Tracing out and integrating will then yield that the slope of $F$ with respect to $\alpha$ is nonpositive. If it vanishes, then so does the second fundamental form, hence the orthogonal complement $F^\perp$ of $F$ with respect to $h$ is holomorphic again, showing the existence of a polystable decomposition $T_X=\oplus F_i$ of $T_X$ into pairwise orthogonal parallel stable subbundles of slope zero. The fact that $c_1(F_i)$ vanishes can be derived from the fact that $T_X$ is semistable with respect to {\it any} Kähler class as follows. If say $c_1(F_1)\neq 0$, then there exists a Kähler class $\alpha$ such that $c_1(F_1)\cdot \alpha^{n-1}>0$, which contradicts the inequalities $c_1(F_i)\cdot \alpha^{n-1}\ge 0$ for all $i$ combined with the identity $c_1(T_X)=\sum_i c_1(F_i)$. Finally, since $\nabla$ is torsion-free and $F_i$ is parallel, the latter is automatically stable under Lie bracket.

\emph{Singular case.} When $X$ is singular, one would like to do the same reasoning over $\Xr$ but unfortunately this does not seem to be possible because of our lack of understanding of the behavior of $\omega$ near the singularities. Instead, one has to work on a resolution of singularities with smooth, approximate Kähler-Einstein metrics and control the error term along the approximation 

\emph{Uniqueness.} The last property to check is that the decomposition is unique up to permutation of the factors (and not just up to isomorphism), which is a bit subtle. First, observe that parallel transport induces a one-to-one correspondence between $G$-invariant subspaces of $V$ and parallel subsheaves of $T_X$, so we can translate the problem in terms of decompositions of the $G$-representation $V$. Then, the key point is to show that given an orthogonal decomposition $V:=V_0 \oplus \oplus_{i=1}^k V_i$ into irreducible representations, it is unique up to permutation of the factors. Here $V_0=V^G$ and the $V_i$'s are nontrivial irreducible representations of the unitary group $G$. This property is a consequence of the fact that $G$ itself splits accordingly, that is $G=\prod_{i=1}^k G_i$ where $G_i \subset \U(V_i)$ by \cite[Theorem~10.38]{Besse}. From there, uniqueness of the decomposition follows easily, see e.g. \cite[Remark~2.10]{GGK}.
\end{proof}

\subsection{Bochner principle}
As we saw in the introduction, the Bochner principle is a key tool in the decomposition theorem in the smooth setting. That latter states that any holomorphic tensor $\sigma \in H^0(X, T_X^{\otimes p}\otimes \Omega_X^{\otimes q})$ is parallel with respect to any Kähler Ricci flat metric $\omega$, that is $\nabla \sigma =0$. The proof is a straightforward consequence of the Bochner identity
\[\Delta_\omega |\sigma|^2 = |\nabla \sigma|^2\]
since the LHS integrates to zero by Stokes formula. Although the proof does not go through as such in the singular setting, the statement itself remains valid. 

\begin{thm}[Bochner principle]
\label{bochner}
In Setup~\ref{setup}, let $\sigma \in H^0(\Xr, T_X^{\otimes p}\otimes \Omega_X^{\otimes q})$ be a holomorphic tensor, where $p,q\ge 0$ are nonnegative integers. Then, one has 
\[\nabla \sigma = 0 \quad \mbox{on } \Xr.\]
In particular, there is a one-to-one correspondence between hololomorphic tensors on $\Xr$ of type $(p,q)$ and $G$-invariant vectors in $V^{\otimes p} \otimes (V^*)^{\otimes q}$. 
\end{thm}

From a technical point of view, the proof resembles that of Theorem~\ref{thm poly} but some non-trivial additional inputs are necessary. It was proved first somehow indirectly in \cite[Theorem~8.2]{GGK} in the projective case and a simplified proof encompassing the Kähler case was later provided in \cite[Theorem~A]{CGGN}. We refer to \cite[Remark~3.5]{CGGN} for a comparison of the two proofs.

\subsection{\'Etale triviality of the Albanese map}

 Let $X$ be a normal compact variety with rational singularities, let $\pi:\wX\to X$ be a resolution of singularities, and set $\mathrm{Alb}(X):=\mathrm{Alb}(\wX)$. If $F$ is a fiber of $\pi$, then $H^1(F,\Z)=0$ hence the restriction of the Albanese map of $\wX$ to $F$ lifts to the universal cover of $\mathrm{Alb}(\wX)$, hence is constant.  Therefore the Albanese map of $\wX$ factors through $\pi$, which defines the Albanese map 
 \[\alpha :X\to \mathrm{Alb}(X)\] whose universal property is that any map from $X$ to a complex torus factors through $\alpha$. We refer to \cite[Theorem~4.3]{Graf18} and the references therein for the details.  Thanks to Theorem~\ref{thm extension forms}, we have 
 \[\dim \mathrm{Alb}(X)=h^0(X, \Omega_X^{[1]})=h^1(X,\cO_X)=:q(X).\]
 A first straightforward consequence of the Bochner principle is that if $c_1(X)=0$, then $q(X) \le \dim X$. In particular, the augmented irregularity of $X$ 
 \[\widetilde q(X):=\sup \{h^0(Y, \Omega_Y^{[1]}), \,\,Y\to X \, \mbox{quasi-étale}\}\]
 is finite, bounded above by $\dim X$. 
 
  We can actually say much more about the Albanese map, if one assumes additionally that $X$ has canonical singularities.  
 
 \begin{thm}[Albanese splits after base change]
 \label{alb}
 Let $X$ be a compact Kähler variety with canonical singularities such that $c_1(X)=0\in H^2(X,\R)$. Then there exists a finite étale cover $f:B\to \mathrm{Alb}(X)$ such that $X\times_A B$ is isomorphic to $F\times B$ over $B$, where $F$ is connected. In particular, $\alpha$ is a surjective analytic fiber bundle with connected fibers. 
 \end{thm}

The above theorem is proved in \cite[Corollary~8.4]{Kawa85} when $X$ is projective and in \cite[Theorem~4.1]{CGGN} when $X$ is merely Kähler. We refer to the remark below the theorem in {\it loc. cit.} for earlier partial results. 

\begin{proof}
By the universal property of $\alpha$, any element $g\in \Aut(X)$ induces an automorphism $\varphi_g$ of $\Alb(X)$ and, moreover, the linear part of $\varphi_g$ is the identity if $g\in \Aut^\circ(X)$, hence $\varphi_g$ is a translation by an element $\varphi(g)\in \Alb(X)$. We thus have a morphism of complex Lie groups $\varphi: \Aut^\circ(X)\to \Alb(X)$. Since $X$ has only canonical singularities, $X$ is not uniruled. Let us briefly recall this standard argument. If $X$ were uniruled, then so would be a resolution of singularities $\widetilde X\to X$. Up to blowing up further, one can assume that $\widetilde X$ admits a fibration in rational curves. If $F\simeq \mathbb P^1$ is a general fiber, the adjunction formula shows that $K_{\widetilde X}$ effective implies $K_F$ effective, a contradiction. 

Therefore one can apply \cite[Proposition~5.10]{Fuj78} to see that $\Aut^\circ(X)$ is a complex torus. One can check that the induced map between Lie algebras $d\varphi: H^0(X, T_X)\to H^0(X, \Omega_X^{[1]})^*$ is the canonical pairing map, which is a perfect pairing as a direct consequence of Bochner principle. Hence $\varphi$ is a finite étale cover. 

Moreover, the same computation shows that the differential of $\alpha$ has a maximal rank at any point $x\in \Xr$; in particular $\alpha$ is surjective. Set $A:=\Alb(X)$, $B:=\Aut^\circ(X)$, $F:=\alpha^{-1}(0)$ and $Y:=X\times_A B$. The map
$F\times B \to X\times B$ sending $(x,g)$ to $(g(x),g)$ induces a map 
\[\psi:F\times B\to Y\] 
commuting with the projection to $B$ and which admits the inverse $\psi^{-1}:(x,g)\mapsto (g^{-1}(x), g)$. It therefore remains to show that $F$ is connected. Since $\alpha$ is locally trivial, its Stein factorization $X\to Z\overset{p}{\to} A$ is such that $p$ is étale, hence $Z$ is a torus. By the universal property of $\alpha$, we get a map $q:A\to Z$ which is readily seen to be an inverse of $p$. This concludes the proof of the theorem.   
\end{proof}

As we saw in the proof, it is important to restrict to canonical singularities. Fortunately, the abundance conjecture is known in numerical dimension zero thanks to \cite[Corollary~4.9]{Nak} in the projective case and results of B. Wang in the Kähler case, cf \cite[Corollary~1.18]{JHM2}. More precisely, if $X$ is a compact Kähler variety with klt singularities such that $c_1(X)=0\in H^2(X, \R)$, then $K_X$ is a torsion $\Q$-line bundle. Therefore, one can take an index-one cover \cite[Definition~5.19]{KM}, which yields the following.  

\begin{thm}
\label{abondance}
Let $X$ be a compact Kähler variety with klt singularities such that $c_1(X)=0\in H^2(X, \R)$. Then there exists a quasi-étale cover $f:Y\to X$ such that $Y$ has canonical singularities and $K_Y\sim \cO_Y$. 
\end{thm}

So if we summarize, starting from $X$ as in Setup~\ref{setup}, one can find a quasi-étale cover $Y\to X$ with canonical singularities thanks to Theorem~\ref{abondance}. If a quasi-étale cover $Y'\to Y$ admits non-zero holomorphic $1$-forms (that is, $\widetilde q(Y)\neq 0$), then a further quasi-étale cover $Y''\to Y$ splits off a torus by Theorem~\ref{alb}, that is $Y''=T\times Z$ where $T$ is a complex torus. Iterating the process with $Z$, we end up with the following statement. 

\begin{prop}[Torus cover]
\label{torus cover}
Let $X$ as in Setup~\ref{setup}. There exists a quasi-étale cover $X'\to X$ such that
\[X'\simeq T\times Z\]
where $T$ is a complex torus, $Z$ has canonical singularities with $K_Z\sim \cO_Z$  and $\widetilde q(Z)=0$.
\end{prop}

\section{Holonomy and irreducible factors}
\subsection{Computing the restricted holonomy}
\label{sec rest hol}
In Setup~\ref{setup}, we have two decompositions
\[V=V_0\oplus \bigoplus_{i\in I} V_i, \quad \mbox{and} \quad V=W_0\oplus \bigoplus_{j\in J} W_j\]
corresponding to the canonical decomposition of $V$ into irreducible representations for $G$ (resp $G^\circ$) and where $V_0=V^G$ (resp. $W_0=V^{G^\circ}$). Moreover, we have $G=\prod_{i\in I} G_i$ where $G_i \subset \U(V_i)$ by \cite[Theorem~10.38]{Besse} and $G^\circ=\prod_{j\in J} G_j^\circ$ where $G_j^\circ \subset \U(W_j)$ by \cite[Corollary~10.41]{Besse} .

Clearly, $V_0\subset W_0$. Next, since $G^\circ$ is normal in $G$, $G$ preserves $W_0$ and for any $j\in J, g\in G$, $g(W_j)$ is stable and irreducible under the action of $G^\circ$, that is $g(W_j)=W_k$ for some $k=k(g)$. In other words, $G$ permutes the $W_j$ for $j\in J$ hence induces a morphism $s:G/G^\circ \to \mathfrak{S}_J$. The kernel of the map $\pi_1(\Xr)\to \mathfrak{S}_J$ obtained by composing $s$ with \eqref{pi 1 hol} yields a quasi-étale cover 
\[f:X^{\rm hol}\to X\]
 such that the holonomy $G'$ of $f^*\omega$ stabilizes every nontrivial irreducible $G'^\circ$-representation. The cover $X^{\rm hol}\to X$ will be referred to as weak holonomy cover, in accordance with \cite[Section~7.2]{GGK}.

  From now on, we replace $X$ with $X^{\rm hol}$. The conclusion of the discussion is that there is a subset $I_0\subset I$ and a bijection $\tau:J\to I\setminus I_0$ such that
\[W_0=V_0\oplus \bigoplus_{i\in I_0}V_i, \quad \mbox{and} \quad W_j=V_{\tau(j)}, \,\, \forall j\in J.\]

Let $U\subset \Xr$ be a small simply connected neighborhood of $x$. The natural map $\Hol_x(U,\omega)\to \Hol_x(\Xr,\omega)$ yields an isomorphism 
\begin{equation}
\label{real an}
\Hol_x(U, \omega) \simeq \Holo_x(\Xr,\omega)=G^\circ
\end{equation}
since $\omega$ is real analytic \cite[II.Theorem~10.8]{KN}. Next, by \cite[Theorem~10.38]{Besse} (cf also \cite[Proposition~3.2.4]{Joyce}), there is a splitting
\begin{equation}
\label{hol dec}
(U,\omega)\simeq (U_0, \omega_0) \times \prod_{j\in J} (U_j,\omega_j)
\end{equation}
where $(U_0, \omega_0)$ is flat and, writing $x=(x_0, (x_j)_{j\in J})$, the holonomy group $\Hol_{x_j}(U_j,\omega_j)$ is irreducible for any $j\in J$. A priori it is not obvious that the decomposition \eqref{hol dec} is parametrized by the set $J$ previously introduced, but this is a consequence of the identification \eqref{real an} and  \cite[Proposition~3.2.1]{Joyce}. In particular, we find $G_j^\circ \simeq \Hol_{x_j}(U_j,\omega_j)$.

Let $n_j:=\dim_\C U_j$; the Kähler manifold $(U_j, \omega_j)$ is Ricci flat but non locally symmetric (otherwise it would be flat, hence have trivial holonomy). Therefore, Berger-Simons classification \cite{Besse} shows that one has either
\begin{equation}
\label{rest hol}
G_j^\circ \simeq \SU(n_j), \quad \mbox{or} \quad n_j=2m_j \,\, \mbox{and} \,\, G_j^\circ \simeq \Sp(m_j)
\end{equation}
where $\Sp(m_j)=\Sp(n_j,\C)\cap \U(n_j)$ is the unitary symplectic group. 

In summary, the decomposition $V=W_0\oplus \bigoplus_{j\in J} W_j$ is $G$-stable and $G$ splits as $G_0\times \prod_{j\in J}G_j$ where $G_0$ is discrete and the identity component of $G_j$ is nothing but the group $G_j^\circ$ from above. Parallel transport yields parallel subbundles $F_0, (F_j)_{j\in J}$ of $T_{\Xr}$ whose saturation induces reflexive subsheaves $\cF_0, (\cF_j)_{j\in J}$ of $T_X$ such that 
\begin{equation}
\label{pre dec}
T_X=\cF_0\oplus \bigoplus_{j\in J} \cF_j
\end{equation}
whose holonomy will be $G_0$ and $G_j$ respectively. In particular, $\cF_0$ is hermitian flat. Thanks to \eqref{rest hol}, we know $G^\circ$, which is enough to derive interesting consequences like the following result. 

\begin{lem}
\label{sym stable}
In the decomposition~\ref{pre dec} above, let us consider a factor $\cF_j$ with $j\in J$. Then for any integer $k\ge 1$, the sheaf $\Sk\cF_j$ is strongly stable with respect to $\alpha$.
\end{lem}

\begin{proof}
Argue by contradiction. Then, there exists $j\in J, k\in \N$ and a quasi-étale cover $f:Y\to X$ such that $\cG:=f^{[*]}\Sk\cF_j$ is destabilized. Let us first set up some further notation. We pick $y\in Y_{\rm reg}$, set $r:=\mathrm{rk}(\cF_j)$, $V:=(f^*\cF_j)_y$, set $G=\Hol_y(f^{[*]}\cF_j|_{Y_{\rm reg}}, f^*\omega)$ which acts on $V$. Recall that we have $G^\circ\simeq\SU(r)$ or $G^\circ \simeq\Sp(\frac r2)$ since restricted holonomy does not change under quasi-étale cover as we have explained already. 

The proof of Theorem~\ref{thm poly} applied to $\cG$ rather than $T_Y$ shows that one can split $\cG=\cG_1\oplus \cG_2$ into orthogonal, parallel subsheaves with respect to $f^*\omega$. In particular, the $G$-representation $S^kV$ admits an invariant subspace and is not irreducible. The same holds for $G^\circ$ as well, which contradicts the irreducibility of the induced representation $S^k\mathbb C^r$ of the groups $\SU(r)$ or $\Sp(\frac r2)$ if $r$ is even. 
\end{proof}

 To go further, one needs to determine the whole holonomy group $G$. More precisely, we would like to show that $G/G^\circ$ is finite. We will split the problem into the flat and the non-flat part. 

\subsection{The flat factor}

The key result here is that the presence of the flat factor $\cF_0$ in the decomposition \eqref{pre dec} is exclusively due to the existence of $1$-forms in some quasi-étale cover. In other words, and given Proposition~\ref{torus cover}, this means that there exist a complex torus $A$ and a splitting $Y\simeq A\times Z$ of a quasi-étale cover $f:Y\to X$ such that $f^{[*]}\cF_0= \mathrm{pr}_A^*T_A$ as subsheaves of $T_Y$. Let us first give some historical context to the above result. \\

This result was first proved by Druel when $X$ is projective as a consequence of the difficult integrability result \cite[Theorem~1.4]{Dru16}, relying among other things on an arithmetic algebraicity criterion due to Bost \cite{Bost01} for leaves of algebraic foliations defined over a number field. This result was later used by \cite{GGK} in order to show that $G/G^\circ$ is finite. Shortly after that, \cite{CGGN} observed that the result can actually be derived from the structure of the Shafarevich variety for linear Kähler groups \cite{CCE15} provided $X$ admits a maximally étale cover (which, in turn, was established a bit later as a consequence of \cite{Fujino22}). Finally, we would like to mention that Campana \cite{CampBB} observed that the results of \cite{HP} (combined with \cite{Dru16, GGK}, cf Corollary~\ref{cor alg int} below) enable one to leave the flat factor $\cF_0$ alone and integrate the non-flat factors first so that one is left with a variety where the whole tangent sheaf is flat, and one can simply appeal to \cite{GKP16}, cf Corollary~\ref{torus quotient}.\\

This section is devoted to proving the result below following the strategy of \cite[Proposition~6.9]{CGGN}.

\begin{thm}
\label{thm flat}
Let $X$ as in Setup~\ref{setup} and let $\cF_0$ be the flat factor in the decomposition \eqref{pre dec}. We have
\[\mathrm{rk}(\cF_0)=\widetilde q(X).\]
\end{thm}

\begin{proof}
Thanks to Proposition~\ref{torus cover}, one can assume that $X$ has canonical singularities, trivial canonical bundle and it is enough to show that $\widetilde q(X)=0$ implies $r:=\mathrm{rk}(\cF_0)=0$. By Corollary~\ref{mqe}, one can assume that $X$ is maximally quasi-étale. In particular, the representation $\rho:\pi_1(\Xr)\to \U(\C^n)$ associated to $\cF_0$ factors through $\pi_1(X)$. 

Let $\pi:\wX\to X$ be a resolution of singularities of $X$. It satisfies two important properties. 

\begin{enumerate}[label=$(\roman*)$]
\item The natural map $\pi_*: \pi_1(\wX)\to \pi_1(X)$ is an isomorphism \cite{Takayama2003}, and $\rho$ factors through $\pi_1(\wX)$.
\item No étale cover $\wX'\to\wX$ is birational to an algebraic fiber space whose base is of general type of positive dimension. Indeed, first observe that $q(\wX')=0$ (by Theorem~\ref{thm extension forms}) and $\kappa(\wX')=0$ (because the finite part $X'\to X$ of the Stein factorization of $\wX'\to X$ is automatically quasi-étale hence $X'$ has canonical singularities and trivial canonical bundle). Now, the main claim follows from Viehweg's partial solution of the Iitaka conjecture \cite{Vieh83} when $X$ is projective, and its upgrade \cite[Theorem~5.1]{Campana04} when $X$ is merely Kähler. 
\end{enumerate}

Given the linear representation $\rho$ (not subject to any further assumption), \cite[Th\'eor\`eme~6.5]{CCE15} shows that up to choosing another resolution $\wX$, there exists an étale cover $\wX'\to \wX$ such that there exists a holomorphic Shafarevich map $\wX'\to \mathrm{Sh}(X)$ associated to $\rho$ (in the sense of \cite[Définition~2.13]{CCE15}) and, moreover, $\mathrm{Sh}(X)$  is bimeromorphic to a smooth fibration in complex tori over a variety $Z$ of general type. Thanks to $(ii)$ above, $Z$ has to be a point and $\mathrm{Sh}(X)$  is bimeromorphic to complex torus. Since $q(\wX')=0$, this complex torus is itself a point.  This implies that the composition $\pi_1(\wX')\to \pi_1(\wX)\to \U(\C^n)$ is trivial, hence $\mathrm{im}(\rho)$ is finite. This means that there exists a quasi-étale cover $f:Y\to X$ such that $f^{[*]}\cF_0\simeq \cO_Y^{\oplus r}$. Since $q(Y)=0$, we have $r=0$, hence the theorem.  
\end{proof}

\begin{rem}
\label{rem flat}
The proof of the theorem shows more generally that if $X$ is as in Setup~\ref{setup} and satisfies $\widetilde q(X)=0$, then any  representation $\pi_1(\Xr)\to \GL(r,\C)$ has finite image. Equivalently, any flat vector bundle $\cF$ on $\Xr$ becomes trivial after a finite étale cover. 
\end{rem}

\subsection{The non-flat factors}
We can now separately analyze the holonomy of the non-flat factors appearing in \eqref{pre dec}.
Let us first state the following intermediate result, which is extracted from \cite[Lemma~6.1]{CGGN}.

\begin{lem}
\label{deligne}
Let $X$ as in Setup~\ref{setup} such that $q(X)=0$. Then 
\[H^1(\Xr, \C)=0.\]
\end{lem}

\begin{proof}
We appeal to a theorem of Deligne \cite[Theorem~8.35 b]{VoisinI}. More precisely, if $\wX\to X$ is a log resolution with exceptional divisor $E$ which is an isomorphism over $X_{\rm reg}$, then we have the exact sequence
 \[0\to H^0(\wX, \Omega_{\wX}^1(\log E)) \to H^1(\wX\setminus E, \C)\to H^1(\wX, \cO_{\wX})\to 0.\]
 Since $H^0(\Xr, \Omega_{\Xr}^1)=q(X)=0$, the first term is zero. Similarly, the last term vanishes by Corollary~\ref{hodge duality}. Therefore the middle term $H^1(\wX\setminus E, \C)=H^1(\Xr,\C)=0$, and the result follows easily.
\end{proof}

 We can now prove the result below, which is due to \cite{GGK, CGGN} in the projective and Kähler setting, respectively.

\begin{thm}
\label{thm holonomy}
Let $X$ as in Setup~\ref{setup}, and assume that $\widetilde q(X)=0$. There exists a quasi-étale cover $Y\to X$ and a decomposition
\[T_Y= \bigoplus_{j\in J} \cF_j\]
into orthogonal, parallel subsheaves such that for any $j\in J$, $\cF_j|_{\Xr}$ has holonomy either $\SU(n_j)$ or $\Sp(\frac {n_j} 2)$ where $n_j=\mathrm{rk}(\cF_j)\ge 2$. 
\end{thm}

\begin{proof}
Up to passing to the weak holonomy cover, one can assume that the decomposition \eqref{pre dec} holds. 
We have the following properties.
\begin{enumerate}
\item The flat factor $\cF_0$ does not appear, cf Theorem~\ref{thm flat}.
\item The group $H_1(\Xr,\Z)$ is finite, cf Lemma~\ref{deligne}.
\end{enumerate}
 We claim that $G_j/G_j^\circ$ is abelian. Indeed, one has $G_j\subset N_j$ where $N_j$ is the normalizer of $G_j^\circ$ in $\U(n_j)$. Now, an elementary computation shows that in both cases $G_j^\circ=\SU(n_j)$ or $G_j^\circ=\Sp(n_j/2)$, one has $N_j/ G_j^\circ \simeq \U(1)$, hence the claim. 
 
 \noindent
The surjection 
\[\phi_j:\pi_1(\Xr)\to G_j/G_j^\circ\]
 factors through the finite group $H_1(\Xr,\Z)$, hence $G_j/G_j^\circ$ is finite by (2). The kernel of $\phi_j$ induces a quasi-étale cover $f_j:Y\to X$ where the holonomy of $f_j^{[*]}\cF_j$ is connected, hence the result. 
\end{proof}

\section{Irreducible pieces of the decomposition}
\label{sec irr}

Before we introduce the notion of irreducible Calabi--Yau and holomorphic symplectic varieties, we would like to illustrate with a very simple example that in the singular case deciding what should be called "irreducible" may not be completely obvious.

\begin{exa}
\label{sing kummer}
Let $X:=A/\{\pm 1\}$ be a singular Kummer surface, where $A$ is an abelian surface. In particular, $X$ is a projective surface with canonical singularities and trivial canonical bundle. One the one hand, one has
\begin{enumerate}[label=\arabic*.]
\item $\pi_1(X)=\{1\}$.
\item $q(X)=0$. 
\item The minimal resolution of $X$ is a K3 surface. 
\item $X$ can be smoothed out by K3 surfaces (at least when $A$ is principally polarized). 
\end{enumerate}
On the other hand, one has 
\begin{enumerate}[label=\arabic*.']
\item $\pi_1(\Xr)$ is an extension of $\Z/2\Z$ by $\Z^4$. 
\item $\widetilde q(X)=2$. 
\item A quasi-étale cover of $X$ is a torus. 
\item $T_{\Xr}$ is flat and the holonomy of any singular Ricci flat metric is $\Z/2\Z$. 
\end{enumerate}
\end{exa}

This example tells us that the geometry (or, roughly speaking the type) of a given variety can drastically change when one passes to a quasi-étale, {\it non-étale} cover. This suggests that whether one looks at $\pi_1(X)$ or $\pi_1(\Xr)$ will provide vastly distinct geometric insight into $X$. Many more relevant examples of singular varieties with trivial canonical bundle have been considered in the past; a thorough account was provided in \cite[Section~14]{GGK} with a view towards classification. 

It comes out of that discussion that it may be reasonable to allow quasi-étale covers and consider $\pi_1(\Xr)$. A possible way to extend the definition of ICY (resp. IHS) manifold to the singular setting would be to require $\Xr$ to be simply connected and have $\bigoplus_{p=0}^n H^0(\Xr, \Omega_X^p)$ be isomorphic to the algebra of an ICY (resp. IHS) manifold. Unfortunately and despite major progress like Theorem~\ref{thm braun}, it is still out of reach to show that the fundamental group of $\Xr$ (for $X$ as in Setup~\ref{setup}) is virtually abelian, and finite when $\widetilde q(X)=0$. 

For that reason, it was suggested by \cite{GKP} that in order to define the singular analogue of ICY and IHS varieties, one could replace the topological assumption on $\pi_1(\Xr)$ by a geometric property to be checked on any quasi-étale cover of $X$. More precisely, the following definitions were proposed by \cite{GKP}. 
 
 \begin{defi}[ICY and IHS varieties]
 \label{defi icy ihs}
 Let $X$ be a compact normal Kähler variety of dimension $n\ge 2$ with canonical singularities such that $K_X\sim \cO_X$. We say that $X$ is 
 \begin{enumerate}[label=$\bullet$]
\item \emph{irreducible Calabi--Yau} (ICY for short) if $n\ge 3$ and for any quasi-étale cover $Y\to X$, one has 
\[\forall p=1, \ldots, n-1, \quad H^0(Y,\Omega_Y^{[p]})=\{0\}.\] 
\item \emph{irreducible holomorphic symplectic} (IHS for short) if there exists a holomorphic symplectic form $\sigma \in H^0(X,\Omega_X^{[2]})$ such that any quasi-étale cover $f:Y\to X$, we have an isomorphism of algebras 
\[\bigoplus_{p=0}^n H^0(Y,\Omega_Y^{[p]})=\C[f^{[*]}\sigma].\] 
\end{enumerate}
\end{defi}
By definition, $\sigma$ is symplectic means that $\sigma|_{\Xr}$ is symplectic in the usual sense. Also, the notation $f^{[*]}\sigma$ refers to the reflexive pull-back; indeed, $f^*\sigma|_{\Xr}$ is a holomorphic $2$-form on a Zariski open subset of $Y$ whose complement has codimension at least two, hence it uniquely extends to a global section of $\Omega_Y^{[2]}$ by definition of the latter sheaf. 

When $X$ is smooth, it follows from the decomposition theorem that an ICY manifold in the sense of Definition~\ref{defi icy ihs} has finite fundamental group (hence has an étale cover which is ICY in the "usual" sense) while an IHS manifold in the sense of Definition~\ref{defi icy ihs} is automatically simply connected, hence is IHS in the "usual" sense too. \\

In the past, many different definitions of what a singular holomorphic symplectic variety could be were proposed, cf \cite[Section~14]{GGK} for a discussion on that topic, or Section~\ref{PSV} below for the connection with the notion of primitive symplectic variety. More recently, \cite{GPP24} discussed the surface case at length while \cite{BGMM25} provided a whole host of old and new examples coming out of performing terminalization of quotients of IHS manifolds. Finally, let us point out to \cite{PR18} and \cite{Sacca23} who give interesting examples of IHS varieties not admitting a symplectic resolution; these example arise from the study of holomorphic moduli spaces of sheaves and Bridgeland-stable objects.\\

The following proposition shows that ICY and IHS varieties can alternatively be defined solely in terms of their holonomy with respect to a singular Kähler Ricci flat metric. 

\begin{prop} 
\label{irreducible}
Let $X,\omega$ be as in Setup~\ref{setup}. We have
\begin{enumerate}[label=$(\roman*)$]
\item $X$ is an ICY variety if, and only if $\Hol_x(\Xr, \omega)=\SU(n)$. 
\item $X$ is an IHS variety if, and only if $\Hol_x(\Xr, \omega)=\Sp(\frac n2)$.
\end{enumerate}
Moreover, in either case $T_X$ is strongly stable with respect to any Kähler class. 
\end{prop}

\begin{proof}
Assume that $\Hol_x(\Xr, \omega)=\SU(n)$. Since $\SU(n)$ leaves a non-zero $(p,0)$ form invariant iff $p=0,n$, Bochner principle implies $K_X$ is trivial (hence $X$ has canonical singularities) and $h^0(X,\Omega_X^{[p]})=0$ for $0<p<n$. Moreover, $\Hol_x^\circ(\Xr,\omega)$ is invariant under quasi-étale cover. Since the latter is already connected, it follows that $X$ is an ICY variety. The same arguments shows that $X$ is an IHS variety provided that $\Hol_x(\Xr, \omega)=\Sp(\frac n2)$.

Next, assume that $X$ is an ICY variety. By Bochner principle, we have $\Hol_x(\Xr, \omega)\subset \SU(n)$ so that we are left to showing the reverse inclusion. It is enough to show that some quasi-étale cover has holonomy $\SU(n)$. Let $Y\to X$ be the cover from Theorem~\ref{thm holonomy}. Since each piece of the decomposition has trivial determinant, there can only be one such piece, otherwise $Y$ would carry non-zero holomorphic forms of degree less than $n$. We thus have to eliminate the possibilities that $T_Y$ has holonomy $\Sp(\frac n 2)$.  But this cannot happen since $Y$ carries no non-zero holomorphic $2$-form.

Finally, assume that $X$ is an IHS variety. By Bochner principle, we have $\Hol_x(\Xr, \omega)\subset \Sp(\frac n2)$ hence we have to show the reverse inclusion. It is enough to show that some quasi-étale cover has holonomy $\Sp(\frac n2)$. Let $Y\to X$ be as above and let $\sigma_Y$ be a non-zero $2$-form on $Y$.  By Bochner principle and since the holonomy group of $(Y,f^*\omega)$ is of the form $\prod_{j\in J}G_j$, $\sigma_Y$ can be written as 
\[\sigma_Y=\sum_{j\in J} \sigma_j\]
 where $\sigma_j\in H^0(Y, \Lambda^{[2]} \cF_j^*)$. Since $\sigma_Y$ is unique up to scaling, there is exactly one non-zero summand in the above decomposition. Since $\sigma_Y$ is symplectic, no summand can vanish. Therefore $|J|=1$ and the holonomy is the symplectic group. 
 
 Let us now prove the last statement. Let $X$ be an ICY or IHS variety. If $X$ is not strongly stable with respect to some Kähler class $\alpha$, then there exists a quasi-étale cover $f:Y\to X$ such that $T_Y= \cG_1\oplus \cG_2$ for some non-zero subsheaves $\cG_1, \cG_2$ with zero slope wrt $f^*\alpha$. By Bochner principle (or, more precisely, the proof of Theorem~\ref{thm poly}), the subbundles $\cG_i|_{Y_{\rm reg}}$ are parallel with respect to $f^*\omega$, hence the holonomy of that metric would be contained in $\U(n_1)\times \U(n_2)$ where $n_i:=\mathrm{rk}(\cG_i)$, contradicting the results we just proved.
\end{proof}

\begin{prop}
\label{prop strongly stable}
Let $X$ as in Setup~\ref{setup}. If $T_X$ is strongly stable with respect to a given Kähler class, then there exists a quasi-étale cover $Y\to X$ such that $X$ is either an ICY or an IHS variety. 
\end{prop}

\begin{proof}
We fix the Kähler class $\alpha$ and work with the associated singular Kähler Ricci flat metric $\omega$. If $\widetilde q(X)\neq 0$, then we get a contradiction Proposition~\ref{torus cover}. Therefore, we can consider the cover $f:Y\to X$ from Theorem~\ref{thm holonomy}. Since $T_Y$ is stable with respect to $f^*\alpha$, we have $|J|=1$ and the statement follows from Proposition~\ref{irreducible}. 
\end{proof}

Let us finish this section by pointing out 

\section{Splitting off the factors in the tangent decomposition}

With Proposition~\ref{torus cover} and Theorem~\ref{thm holonomy} in hand, the remaining task to accomplish is to show if $X$ has vanishing augmented irregularity, then each factor in the decomposition $T_Y= \bigoplus_{j\in J} \cF_j$ of the quasi-étale cover $Y\to X$ corresponds to direct factor of the variety $Y$ itself, maybe up to a further quasi-étale cover. That is, we need to show that there exist a quasi-étale cover $f:Z\to Y$ and a decomposition $Z\simeq   \prod_{j\in J} Z_j$ such that $\mathrm{pr}_{Z_j}^*T_{Z_j}=f^{[*]}\cF_j$ for any $j \in  J$. Indeed, if that were the case then Proposition~\ref{irreducible} would guarantee that $Z_j$ would be either ICY or IHS.

At this point, the road forks between the projective and the Kähler setting. So far all the results stated were valid in the Kähler case, but in order to go further, one needs at this precise point to work in the algebraic category. This difficult task to split $X$ into a product can be reached in the following two steps. 
\begin{enumerate}[label=(\Roman*)]
\item Show that the foliations $\cF_j$ are algebraically integrable. 
\item Show that a decomposition $T_X=\bigoplus \cF_j$ into algebraically integrable foliations induces a splitting of a quasi-étale cover into as many factors. 
\end{enumerate}

\subsection{Algebraic integrability of $\cF_j$.}
Let us start by recalling the notion of pseudoeffective vector bundle or sheaf on a projective variety. It will play a key role in all what follows. 

\begin{defi}
\label{def psef}
Let $X$ be a normal projective variety, let $H$ be an ample Cartier divisor and let $\cF$ be a coherent reflexive sheaf. One says that $\cF$ is pseudoeffective (psef for short) if
for any $c>0$, there are integers $i,j>0$ satisfying $i>cj$ and  $H^0(X, \SiF \otimes \cO_X(jH))\neq 0.$
\end{defi}

Let us first make a couple of remarks about this definition. 

$\bullet$ If $\cF$ is a line bundle, then the condition above expresses that $c_1(\cF)$ is a limit of effective classes which is the usual definition of pseudoeffectivity. 

$\bullet$ If $\cF$ is locally free, then $\cF$ is psef iff the tautological quotient line bundle $\cO_{\PP(\cF^*)}(1)$ is psef where $\pi:\PP(\cF^*)\to X$ is the projective bundles of lines in $\cF^*$, cf \cite[Lemma~2.7]{Dru16}. The main point is that $\pi_*\cO_{\PP(\cF^*)}(k)\simeq S^kF$ and $\cO_{\PP(\cF^*)}(1)$ is relatively ample. 

$\bullet$ In the general case, $\PP(\cF^*)$ is too singular but one may replace it with a suitable modification in order to get a similar statement as in the locallly free case, cf \cite[Lemma~2.3]{HP}.\\

We are now going to state a fundamental result asserting that a foliation $\cF$ is algebraically integrable as soon as $\cF^*$ is {\it not} psef. This result has long history and goes back to Bogomolov-McQuillan \cite{bogomolov_mcquillan01} and Bost \cite{Bost01} in an arithmetic context. Actually their result also provides a criterion for the closure of the leaves to be rationally connected, but we will leave out this part which is not relevant in our context.  The following statement is extracted from \cite[Proposition~8.4]{Dru16}. 

\begin{thm}[Criterion for algebraic integrability]
\label{BM}
Let $X$ be a normal projective variety and let $\cF\subset T_X$ be a foliation. If $\cF^*$ is not pseudoeffective, then $\cF$ has algebraic leaves.  
\end{thm}

We refer to \cite[Section~4.1]{CP19} for a very nice exposition of the proof along with important generalizations. \\

The strategy to prove Step (I) should now be rather clear to the reader: we need to show that the foliations $\cF_j^*$ are not psef. In light of Lemma~\ref{sym stable}, the following theorem of Höring and Peternell \cite[Theorem~1.1]{HP} essentially provides the last missing piece.

\begin{thm}
\label{thm HP}
Let $X$ be a projective normal variety of dimension $n$ which is smooth in codimension two, let $H$ be an ample Cartier divisor  and let $\cF$ be a reflexive sheaf such that 
\begin{enumerate}[label=$(\roman*)$]
\item $c_1(\cF)\cdot H^{n-1}=0$. 
\item $\SkF$ is $H$-stable for every integer $k\ge 1$.
\end{enumerate} 
If $\cF$ is pseudoeffective, then 
\[c_1(\cF)^2\cdot H^{n-2}=c_2(\cF)\cdot H^{n-2}=0.\] 
\end{thm} 

\begin{rem}
\label{rem HP}
 Let us make a few comments. 

$1.$ This result fits in with the conjecture of Pereira and Touzet \cite{PT13}. We refer to \cite{CCP21} for analogous results in the smooth Kähler setting where only a small number of symmetric powers are required to be stable and to \cite[Theorem~6.1]{Dru16} for a previously established version of the theorem holding when the rank of $\cF$ is no greater than $3$. 

$2.$ Since $X$ is smooth in codimension two, the quantity $c_2(\cF)\cdot H^{n-2}$ is well-defined (e.g. one can define it unambiguously on any log resolution of singularities). Furthermore, if $H$ is very ample and $H_i\in |H|$ are general, then $S:=H_1\cap \cdots \cap H_{n-2}$ is smooth and the Chern number above coincides with $c_2(\cF|_S)$.  

$3.$ If $X$ has klt singularities and $c_1(\cF)^2\cdot H^{n-2}=c_2(\cF)\cdot H^{n-2}=0$, then $\cF|_{\Xr}$ is flat by \cite[Theorem~1.20]{GKP16}. The proof goes roughly as follows. Without loss of generality, one can assume that $X$ is maximally quasi-étale (i.e. the identity map of $X$ is maximally quasi-étale in the sense of Definition~\ref{mqe}). By the usual Kobayashi-Hitchin correspondence, $\cF|_S$ is (hermitian) flat, given by a representation $\pi_1(S)\to \U(r)$ where $r=\mathrm{rk}(\cF)$. By Goresky-MacPherson's version of Lefschetz's hyperplane theorem for singular varieties \cite[Section II.1.2]{GMbook}, the natural map $\pi_1(S)\to \pi_1(\Xr)$ is an isomorphism. Hence we get a flat bundle  $\cG$ on $\Xr$ extending $\cF|_S$; since $X$ is maximally quasi-étale, $\cG$ extends to a locally free flat sheaf on $X$. In particular all its Chern classes vanish, and all such extensions $\cG$ when $S$ moves live in a bounded family. This allows one to use a Bertini-type argument and conclude that $\cG\simeq \cF$ provided that $S$ is general. 
\end{rem}

\begin{proof}[A few elements of the proof of Theorem \ref{thm HP}]
In order to explain some ideas more clearly, let us assume that $\cF$ is locally free. Let $\pi : \PP(\cF^*)\to X$ be the projectivized bundle and let $\zeta:=c_1(\cO_{\PP(\cF^*)}(1))$, which is psef by assumption. \\

Let $S$ be a general complete intersection surface. If one can show that $\zeta_S:=\zeta|_{\pi^{-1}(S)}$ is nef, then $\det \cF|_S$ is nef too and the assumption that $c_1(\cF|_S)\cdot H|_S=0$ will imply $c_1(\cF|_S)=0$. Hence $\cF|_S$ is nef with numerically trivial determinant, i.e. $\cF|_S$ is numerically flat and $c_2(\cF|_S)=0$ by \cite[Corollary~1.19]{DPS94} so that the theorem would follow. \\

In order to show that $\zeta_S$ is nef, it is equivalent by \cite[Théorème~2]{Paun98} to show that the class $\zeta_S$ is nef in restriction to each irreducible component $W$ of the non-nef locus of $\zeta$. Now, if $\pi(W)$ has codimension at least two in $X$, then $W\cap \pi^{-1}(S)$ is contained in finitely many fibers of $\pi$ in restriction to which $\zeta$ is ample. So one reduces to showing 
\[\codim_X(\pi(W))\ge 2.\]

The proof of the above inequality is by descending induction on the dimension of $W$. We will sketch the argument when $\dim W=\dim \PP(\cF^*)-1$ which is the first nontrivial case. Argue by contradiction and pick such a $W$ with $\codim_X(\pi(W))=1$. Let $C$ be a general complete intersection curve. The bundle $\cF|_C$ is stable with trivial determinant on a curve, it is hermitian flat (e.g. by the Kobayashi--Hitchin correspondence). In particular, $\zeta_C$ is nef. The contradiction will follow from the following two key facts.\\

{\it Fact 1.} Since all symmetric powers of $\cF|_C$ are stable, the class $\zeta_C$ is {\it ample} in restriction to any proper subvariety of $\PP(\cF|_C^*)$. This is \cite[Lemma~2.13]{HP}, which generalizes the rank $2$ case due to Mumford. \\

{\it Fact 2.} Let $W$ as above and let $W_C:=W\cap \pi^{-1}(C)$. Then there exists $a>0$ such that
\[\zeta_C^r\ge a \zeta_C^{r-1}\cdot W_C.\]
This is a consequence of \cite[Lemma~3.4]{HP} given that $\zeta_C$ is nef.

The contradiction is now within reach. The LHS in the inequality above is simply $\zeta_C^r=\pi^*c_1(\cF|_C)\cdot \zeta_C^{r-1}=0$ but the RHS is positive by the first fact. We leave out the proof of the two crucial lemmas and refer the interested reader to \cite{HP}.  
\end{proof}

\begin{cor}
\label{cor alg int}
Let $X$ as in Setup~\ref{setup} such that $X$ is smooth in codimension two and satisfies $\widetilde q(X)=0$. Let $\cF_j$ be any factor of the tangent sheaf of the quasi-étale cover $Y\to X$ provided by Theorem~\ref{thm holonomy}. Assume additionally that
\begin{enumerate}[label=$(\roman*)$]
\item $X$ is projective.
\item $X$ is smooth in codimension two.
\end{enumerate}
Then $\cF_j$ has algebraic leaves. 
\end{cor}

\begin{proof}
If $\cF_j$ is not psef, then Theorem~\ref{BM} provides the conclusion. Otherwise, thanks to Lemma~\ref{sym stable} one can appeal to Theorem~\ref{thm HP} and the third item in Remark~\ref{rem HP} to deduce that $\cF_j|_{Y_{\rm reg}}$ is flat. By Remark~\ref{rem flat}, this implies that $\cF_j$ is trivial after étale cover, which contradicts the identity $\widetilde q(X)=0$. 
\end{proof}

\subsection{Splitting a quasi-étale cover}

The main result of this section is the following splitting result for algebraically integrable foliations. It is originally due to Druel, cf \cite[Proposition 4.10]{Dru16} where additional assumptions are made, both on the singularities of $X$ and the positivity of $K_{\mathcal F_i}$. More precisely, in {\it loc. cit.} $X$ is assumed to have canonical singularities, and the $K_{\mathcal F_i}$ are assumed to be $\Q$-linearly trivial. The version below is extracted from \cite[Theorem~C]{DGP}.
\begin{thm}
\label{split druel}
Let $X$ be a normal projective variety with klt singularities and let $$T_X = \bigoplus_{i\in I} \cF_i$$ be a decomposition of $T_X$ into 
involutive subsheaves with algebraic leaves. Then there exists 
a quasi-\'etale cover $f\colon Y \to X$ as well as a decomposition 
\[Y \simeq \prod_{i\in I} Y_i\] 
of $Y$ into a product of normal projective varieties
such that the decomposition $T_X = \bigoplus_{i\in I} \cF_i$ lifts to the canonical
decomposition $$T_{\prod_{i\in I} Y_i}= \bigoplus_{i \in I} \textup{pr}_i^*T_{Y_i}.$$
\end{thm}

The result above is difficult and rather technical. We will not attempt to provide a full proof, but rather hint at a few important ideas that underly the general strategy. Before we do so, let us state the following result due to Koll\'ar--Larsen which will be used several times: 

\begin{lem}[{\cite[Prop. 18]{kollar_larsen}}]
\label{lemma:splitting}
Let $X$, $Y$ and $Z$ be normal projective varieties, and let $\pi\colon X \dashrightarrow Y\times Z$ be a birational map that 
does not contract any divisor. Suppose in addition that $X$ has rational singularities. Then 
$X$ decomposes as a product $X\simeq Y' \times Z'$ and there exist birational maps $\pi_Y\colon Y' \dashrightarrow Y$ (resp. $\pi_Z:Z'\dashrightarrow Z$ such that 
$\pi = \pi_Y \times \pi_Z$.
\end{lem}

\begin{rem}
\label{rem split}
When the varieties at stake are merely Kähler, a similar statement holds under the more restrictive assumption that $\pi^{-1}$ is a morphism and $q(X)=0$, cf \cite[Lemma~7.5]{BGL}. For the reader's convenience, we outline below the main ideas of the proof.

Without loss of generality, one can assume that $Y$ and $Z$ are smooth. Moreover, since $H^1(X,\cO_X)=0$, it is not difficult to see that $H^1(Y,\cO_Y)=H^1(Z,\cO_Z)=0$. We consider a general fiber $Y_z=Y\times \{z\}$ and its image $Y''=f(Y_z)$, with normalization $\nu_Y:Y'\to Y''$, and similarly for $Z_y,Z',Z''$. We have a diagram
 \[
\begin{tikzcd}
Y\times Z \arrow[d, "g"] \arrow[r,"f"] & X \\
Y'\times Z'  & 
 \end{tikzcd}
\]
where $g$ is bimeromorphic. It is enough to show that $f$ and $g$ have the same fibers since it will imply that the bimeromorphic map $Y'\times Z' \dashrightarrow X$ is an isomorphism and conclude the proof of the theorem. 

Let $\omega$ be a Kähler class on $X$. By Künneth formula, we have 
\begin{equation}
\label{oab}
f^*\omega=p^*\alpha+q^*\beta
\end{equation} for some classes $\alpha \in H^2(Y,\R)$, $\beta\in H^2(Z, \R)$ and $p,q$ are the projections from $Y\to Z$ to the first and second factor respectively. Recall that $f$ induces maps $Y_z\to Y''$ (resp. $Z_y\to Z''$), so that restricting the identity \eqref{oab} to $Y_z$ (resp. $Z_y$) yields $\alpha = f^*\omega|_{Y''}$ (resp. $\beta=f^*\omega_{Z''}$) under the identification $Y_z\simeq Y$ (resp. $Z_y\simeq Z$) induced by $p$ (resp. $q$). But we have $f|_{Y_z}=\nu_Y \circ g|_{Y_z}$ hence $\alpha=g^*\nu_Y^*\omega|_{Y''}$ and $\beta=g^*\nu_Z^*\omega|_{Z''}$ so that 
\[f^*\omega= g^*(\omega_1+\omega_2)\]
where $\omega_1 = \nu_Y^*\omega|_{Y''}$ and $\omega_2=\nu_Z^*\omega|_{Z''}$ are both Kähler classes on $Y'$ and $Z'$ respectively. From that identity it is clear that $f$ contracts a subvariety $V\subset Y\times Z$ to a point iff $g$ does so too. 
\end{rem}

\bigskip

\begin{proof}[A few elements of the proof of Theorem~\ref{split druel}]

It is enough to consider the case of two foliations $T_X=\cF_1\oplus \cF_2$. For simplicity, we will assume that $X$ has canonical singularities and $K_X$ is linearly trivial.\\

The first reduction, which will be important for the next step, allows one to assume that $X$ has $\Q$-factorial terminal singularities. Indeed, thanks to \cite{BCHM}, there exists a {\it crepant} birational map $\wX\to X$ with $\wX$ having $\Q$-factorial terminal singularities. If there is a quasi-étale cover $\widetilde Y\to \wX$ such that $\widetilde Y \simeq \widetilde Y_1\times \widetilde Y_2$, then the finite part $Y\to X$ of the Stein factorization of $\widetilde Y\to X$ is quasi-étale and has a birational map $\widetilde Y_1\times \widetilde Y_2\to Y$, hence Lemma~\ref{lemma:splitting} allows one to conclude. \\

Since the foliations are algebraically integrable, they are induced by a rational map. More precisely, there is a unique normal complex projective variety $Y_i$ contained in the normalization of the Chow variety of $X$ whose general point parametrizes the closure of a general leaf of $\cF_i$, cf e.g. \cite[Lemma~3.2]{AD}. It induces a rational map $\varphi_i:X\dashrightarrow Y_i$ such that $\cF_i\simeq T_{X/Y_i}$.  Now since $X$ is assumed to have $\Q$-factorial terminal singularities, it is not too difficult to show that $\varphi_i$ is almost proper (i.e. its indeterminacy locus does not dominate $Y_i$), cf \cite[Lemma~3.12]{Dru16}. Moreover, since $\cF_i$ is a regular foliation on $\Xr$, $\varphi_i|_{\Xr}$ is well-defined. \\

Next, a key step is to show that there exist Zariski open sets $Y_i^\circ \subset Y_i^{\rm reg}$ and $X_i^\circ \subset X$ with $\codim_X (X\setminus X_i^\circ)\ge 2$ such that the induced map $\varphi_i|_{X_i^\circ}:X_i^\circ \to Y_i^\circ$ is a projective morphism with irreducible fibers.  This is the content of \cite[Proposition~3.13]{Dru16}; we admit this non-trivial and technical result. We would like to construct rational sections of the maps $\varphi_i$; this will be achieved after a quasi-étale base change as we explain below. \\

Let $F_i$ be a general fiber of $\varphi_i$, and define $F_i^\circ:=F_i\cap \Xr$. Let us consider the normalization $Z_1^\circ$ of $F_2^\circ \times_{Y_1} \Xr$ (and similarly for $Z_2^\circ$). The fibers of $Z_1^\circ\to \Xr$ over $x\in \Xr$ consists in the intersection of $F_2^\circ$ and leaf $\varphi_1^{-1}(x)$ of $\cF_1$ through $x$. Since $\cF_i|_{\Xr}$ are regular foliations, such fibers are finite. Moreover, it is elementary to compute a local normal form in codimension one for the morphisms $F_2^\circ\to Y_1$ and $\Xr\to Y_1$ from which one sees that  $Z_1^\circ\to \Xr$ is quasi-finite and étale in codimension one. That map is actually finite over an open subset of $\Xr$ whose codimension is at least two, hence it extends to a quasi-étale cover $f_1:Z_1\to X$.   

The maps $Z_1^\circ \to F_2^\circ$ and $F_2^\circ \to F_2^\circ \times_{Y_1} \Xr$ induce respectively an almost proper map $\widetilde \varphi_1 : Z_1\dashrightarrow F_2$ and a rational section $F_2\dashrightarrow Z_1$ of $\widetilde \varphi_1$ which image we denote by $G_2$. Let us now consider the composed map $\varphi_2\circ f_1: Z_1\dashrightarrow  Y_2$ whose Stein factorization we denote by $\widetilde \varphi_2: Z_1\dashrightarrow  \widetilde Y_2$.
 \[
\begin{tikzcd}
 Z_1 \arrow[d, dashed, "\widetilde \varphi_1"] \arrow[dr, "f_1"] \arrow[rr, "\widetilde \varphi_2", dashed] & &\widetilde Y_2 \arrow[dr] &\\
 F_2 \arrow[u, dashed, bend left] \arrow[dr] & X \arrow[d, " \varphi_1", dashed] \arrow[rr, "\varphi_2", dashed]  & & Y_2  \\
   & Y_1 & &
 \end{tikzcd}
\]
Clearly, $G_2$ is included in a fiber of $\widetilde \varphi_2$. With a little bit of work, cf \cite[Claim~4.13]{Dru16}, one can show that the inclusion is an equality. A general fiber $\widetilde F_1$ of $\widetilde \varphi_1$ intersect $G_2$ transversely at one point, thus it intersects a general fiber $\widetilde F_2$ of $\widetilde \varphi_2$ in the same fashion. This implies that the map
\[\widetilde \varphi_1\times \widetilde \varphi_2 : Z_1\dashrightarrow F_2\times \widetilde Y_2\]
is birational, and that we have birational identifications $F_2\simeq \widetilde F_2$ and $\widetilde Y_2 \simeq \widetilde F_1$. Now, the varieties $\widetilde F_i$ are terminal with trivial canonical bundle hence the induced birational map
\[Z_1\dashrightarrow \widetilde F_1\times \widetilde F_2\]
is isomorphic in codimension one. The theorem now follows from another application of Lemma~\ref{lemma:splitting}. 
\end{proof}

We have now all we need to prove the decomposition theorem for projective varieties. 

\begin{proof}[Proof of Theorem~\ref{thm BB klt} for projective varieties]
Let $X$ be as in Setup~\ref{setup}; assume additionally that $X$ is projective. As we explained in the first step of the proof of Theorem~\ref{split druel}, Lemma~\ref{lemma:splitting} enables us to reduce to the case where $X$ has terminal singularities; in particular $X$ is smooth in codimension two. 

Thanks to Proposition~\ref{torus cover}, one can assume additionally that $\widetilde q(X)=0$. Let $Y\to X$ be the quasi-étale cover provided by Theorem~\ref{thm holonomy}, with a decomposition $T_Y=\bigoplus_{j\in J} \cF_j$. By Corollary~\ref{cor alg int}, the foliations $\cF_j$ are algebraically integrable. The conclusion of the theorem can now be reached by applying successively Theorem~\ref{split druel} and Proposition~\ref{irreducible}.
\end{proof}

\section{The Kähler case}

In this section, we explain how to deal with the general Kähler case of Theorem~\ref{thm BB klt} by relying on the projective case previously established. This is done in \cite{BGL} and requires a fair amount of work, and we will here only give the outline of the proof along with the main important ideas. 

\subsection{Algebraic approximation}

The starting point is that one can deform locally trivially any $X$ as in Setup~\ref{setup} to a projective variety of the same type, cf \cite[Theorem~A]{BGL}. 

\begin{thm}
\label{alg approx}
Let  $X$ as in Setup~\ref{setup}. There exists a smooth germ $(B,0)$ as well as a locally trivial family $\pi: \cX\to B$ such that 
\begin{enumerate}[label=$(\roman*)$]
\item $X_0\simeq X$.
\item There is a sequence $b_i\to 0$ such that $X_{b_i}$ is projective. 
\end{enumerate}
\end{thm}

\begin{rem}
\label{weak kahler}
The family $\pi: \cX\to B$ in the theorem above is in general not Kähler. However, one can show that there is a smooth $(1,1)$-form $\theta$ on $\cX$ such that $\theta|_{X_b}$ is a Kähler form for each $b\in B$. 
\end{rem}

Let us recall some terminology here.

Given a flat proper map $\pi: \cX\to B$, we say that $\pi$ is locally trivial if it is locally analytically trivial on the source. This means that given any $x\in \cX$ there exists an open euclidean neighborhood $U_x\subset \cX$ of $x$ such that if one sets $b:=\pi(x)$ and $V_b:=\pi(U_x)$, there exists a biholomorphism $F_x:U_x\to (U_x\cap X_b) \times V_b$  commuting with the projections to $V_b\subset B$. 

Now, given a flat map $\pi: \cX\to B$, we say that $\pi$ is topologically trivial on the base if for any $b\in B$, there exists an open euclidean neighborhood $V_b$ of $b$ along with an homeomorphism $\pi^{-1}(V_b)\to X_b\times V_b$ commuting with the projections to $V_b\subset B$. 

It is known that a locally trivial family enjoys the following properties:
\begin{enumerate}[label=$\bullet$]
\item $\pi$ is topologically locally trivial {\it on the base}. 
\item $\pi$ admits a simultaneous resolution of singularities, i.e. there exists a bimeromorphic proper surjective map $\widetilde \cX\to \cX$ with $\widetilde \cX$ smooth such that the composition $\widetilde \cX\to B$ is smooth.  
\item $\pi|_{\cX_{\rm reg}}: \cX_{\rm reg}\to B$ is topologically locally trivial on the base. In particular, any quasi-étale cover $Y_b\to X_b$ for some $b\in B$ (close to $0$) can be lifted to a global quasi-étale cover $\cY\to \cX$. 
\end{enumerate}

For proofs, we refer to \cite[Proposition~5.1]{AV19} and Corollary~2.27 and Lemma~3.7 in \cite{BGL}, respectively.

\begin{proof}[About the proof of Theorem~\ref{alg approx}]
If $X$ is smooth, this classical result is proved in two steps. First, one shows that deformations of $X$ are unobstructed (theorem of Bogomolov-Tian-Todorov). In particular, there is a miniversal deformation $\cX\to \mathrm{Def}^{\rm lt}(X)$ over a {\it smooth} germ $\mathrm{Def}^{\rm lt}(X)$ of dimension $H^1(X,T_X)$. Density of the locus of projective fibers is then showed using Green-Voisin criterion, that is one shows that given a Kähler class $\alpha\in H^1(X,\Omega_X^1)$, the composition of the cup product $H^1(X, T_X)\to H^2(X, T_X\otimes \Omega_X^1)$ with the contraction $H^2(X, T_X\otimes \Omega_X^1)\to H^2(X,\cO_X)$ is surjective. 

In the singular case, we use the same strategy, but we are actually unable to show that locally trivial deformation are unobstructed in general. Instead, we only consider special locally trivial deformations with tangent space corresponding tho the subspace $H^1(X, E)\subset H^1(X,T_X)$ where $E$ is the orthogonal $\bigcap_{\sigma \in H^0(X, \Omega_X^{[2]})} \mathrm{ker}(\sigma)$ with respect to a fixed singular Ricci flat Kähler metric. The advantage of working with $E$ is that it has a symplectic structure (which allows one to show unobstructedness) and still satisfies that $H^1(X,E)\to H^2(X,\cO_X)$ is surjective.
\end{proof}

\begin{cor}
\label{cor split away}
Let $X$ as in Setup~\ref{setup} such that $\widetilde q(X)=0$. Assume that no quasi-étale cover of $X$ is ICY or IHS (equivalently, $T_X$ is not strongly stable, cf Proposition~\ref{prop strongly stable}). Up to replacing $X$ with a quasi-étale cover, there exists a locally trivial deformation $\pi : \cX\to \bD$ over the unit disk $\bD\subset \C$ such that 
\begin{enumerate}[label=$(\roman*)$]
\item $X_0\simeq X$.
\item There are two locally trivial deformations $\cY^*\to \bD$ and $\cZ^*\to \bD$ with positive dimensional fibers such that 
\[\cX^*:=\cX\times_{\bD}\bD^* \simeq \cY^*\times_{\bD^*} \cZ^*.\] 
\end{enumerate}
\end{cor}

\begin{proof}[Sketch of proof]
We proceed in several steps. 

Consider the family $\cX\to B$ provided by Theorem~\ref{alg approx} and pick a projective fiber $X_{b_0}$. Since any quasi-étale cover of $X_{b_0}$ can be lifted to a quasi-étale cover of $\cX$ and since the functions $b\mapsto h^0(X_b, \Omega_{X_b}^{[p]})$ are constant (e.g. by Corollary~\ref{hodge duality}), we see that $\widetilde q(X_{b_0})=0$ and that $X_{b_0}$ is not covered by an ICY variety. Indeed, since $T_X$ is not strongly stable, Theorem~\ref{thm holonomy} yields non-trivial intermediate degree holomorphic forms on a cover of $X$. To exclude that $X_{b_0}$ is not (covered by) a IHS variety is a bit more subtle than a form counting argument and we refer to \cite[Corollary~3.10]{BGL}.\\

Therefore, up to passing to a cover, one gets by the projective case of Theorem~\ref{thm BB klt} a splitting $X_b=Y_b\times Z_b$ for some $b\in \mathbb D^*$, with $Y_b, Z_b$ having zero augmented irregularity. In particular, we have by Corollary~\ref{hodge duality} and Bochner principle the identity $h^1(Y_b, \cO_{Y_b})=h^0(Y_b, T_{Y_b})=0$ and similarly for $Z_b$. By Künneth formula, this implies that the map $\Dlt(Y_b)\times \Dlt(Z_b)\to \Dlt(X_b)$ is an isomorphism (which can be checked at the level of the tangent spaces and of the obstruction spaces). In particular, we get two families $\cY_U, \cZ_U$ over small euclidean neighborhood $U$ of $b\in \bD$ such that $\cX\times_{\bD}U \simeq \cY_U\times_\bD \cZ_U$.\\

The last step consists in extending the families $\cY_U$ and $\cZ_U$ to a Zariski neighborhood of $b\in \mathbb D$, maybe up after passing to a finite cover on the base. This relies on Douady spaces. More precisely, a computation of the dimension of the tangent space of the Douady space $\cD(X_b)$ at the point $[Y_b]\times \{z_b\}$ for $z_b\in Z_b^{\rm reg}$ shows that the latter is smooth of dimension $\dim Z_b$. In particular, there is a unique irreducible component $D_{Y_b}\subset \cD(X_b)$ passing through $[Y_b]\times \{z_b\}$. We can actually repeat the same operation but working over the base $B$ with relative Douady spaces $\cD(\cX/B)$ and find the unique component $D_Y\subset \cD(\cX/B)$ passing through that same point. A crucial input is the properness over $B$ of irreducible components of $\cD(\cX/B)$ whose general element parametrizes pure dimensional reduced subspaces, cf \cite[Proposition~4.2]{BGL}. This is essentially classical \cite{Fuj_Douady} except that $\cX\to B$ is only weakly Kähler in the sense of Remark~\ref{weak kahler}. In particular, $D_Y\to B$ is proper, and we can consider the Stein factorization $D_Y^\nu \to B_Y\to B$ where the last map is finite.  One can do the same thing for $\cZ_U$ and get $D_Z^\nu \to B_Z\to B$. 

The family $\cX\times_B U\to \cZ_U$ is the pullback of the universal family $\cF_Y\subset D_Y\times_B \cX\to D_Y$ by some injective map $\cZ_U\to D_Y$. It is not too hard to infer that the evaluation map $\cF_Y\to \cX$ becomes an isomorphism after taking the fiber product with some Zariski open subset of $B_Y$. Therefore we get  a Zariski open subset $B'\subset B_Y\times_B B_Z$ and two projections maps $\cX\times_B B' \to D_Y^\nu\times_{B_Y} B'$ and $\cX\times_B B' \to D_Z^\nu\times_{B_Z} B'$ which induce an isomorphism
\[\cX\times_{B}B' \to (D_Y^\nu\times_{B_Y} B') \times_{B'}(D_Z^\nu\times_{B_Z} B')\]
over a Zariski open subset of $B'$ and one can just pick a general one dimensional disk through the special point to get the desired family; we refer to \cite[Lemma~1.4]{BGL} for details. 
\end{proof}

\subsection{Splitting of the singular Ricci flat Kähler metric}
Let $X$ as in Corollary~\ref{cor split away} such that $K_X\sim \cO_X$. 

\subsubsection{Splitting of the tangent sheaf}
The family $\pi : \cX\to \bD$ is topologically trivial hence the local systems $R^k\pi_*\Q_{\cX}$ are constant. Let $\pi_1: \cY^*\to \bD^*$ and $\pi_2:\cZ^*\to \bD^*$ be the two families constructed in the corollary. Over $\bD^*$, we have a splitting
\[ R^k\pi_*\Q_{\cX^*}=\bigoplus_{p+q=k} R^p {\pi_1}_*\Q_{\cY^*} \otimes R^q {\pi_2}_*\Q_{\cZ^*}\]
Since $R^k\pi_*\Q_{\cX^*}$ and ${\pi_2}_*\Q_{\cZ^*}$ have no monodromy, neither does $R^k {\pi_1}_*\Q_{\cY^*}$. Similarly, ${\pi_1}_*\Q_{\cY^*}$ having no monodromy implies that $R^k {\pi_2}_*\Q_{\cZ^*}$ has no monodromy either. In particular, the inclusion $j:\bD^*\to \bD$ yields an isomorphism of local systems
\begin{equation}
\label{loc syst dec}
 R^k\pi_*\Q_{\cX}=\bigoplus_{p+q=k} j_*R^p {\pi_1}_*\Q_{\cY^*} \otimes j_*R^q {\pi_2}_*\Q_{\cZ^*}.
 \end{equation}
In \cite[Proposition~5.1]{BGL}, it is explained how one can infer from \eqref{loc syst dec} that if $m=\dim \cY^*-1$ then $j_*{\pi_1}_*\Omega_{\cY^*/\bD^*}^{[m]}$ is a sub-vector bundle of $\pi_*\Omega_{\cX/\bD}^{[m]}$. In particular, one can find a relative reflexive $m$-form $\tau_Y \in H^0(\cX, \Omega_{\cX/\bD}^{[m]})$ which specializes to a trivialization of $K_{Y_t}$ for any $t\neq 0$ and does not identically vanish on $X$. One can perform the same argument to get $\tau_Z$ and see that $\tau_Y\wedge \tau_Z$ is a trivialization of $K_{\cX/\bD}$. Defining $A:=\mathrm{rad}(\tau_Z)$ and $B:=\mathrm{rad}(\tau_Y)$ yields a decomposition
\begin{equation}
\label{tangent split}
T_{\cX/\bD} \simeq A\oplus B
\end{equation}
into foliations such that $A|_{\cX^*}\simeq \mathrm{pr}_{\cY^*}^*(T_{\cY^*/\bD^*})$ and $B|_{\cX^*}\simeq \mathrm{pr}_{\cZ^*}^*(T_{\cZ^*/\bD^*})$. \\

By construction, the foliations $A_0, B_0$ have compact leaves and it is conceivable that this property alone would be enough to split (a cover of) $X_0$ into two pieces like in the projective case, cf Theorem~\ref{split druel}. However, due to the absence of a Kähler replacement of the latter result, we will in the following rely heavily on the fact that $X_t$ is split for $t\neq 0$. 

\subsubsection{Cohomological considerations}

Via the injective map $H^1(X_t, \PH_{X_t})\to H^2(X_t, \R)$ and the isomorphism $ R^2\pi_*\R_{\cX}=j_*R^2 {\pi_1}_*\R_{\cY^*} \oplus j_*R^2 {\pi_2}_*\R_{\cZ^*}$ obtained in \eqref{loc syst dec}, the fiberwise Kähler class $[\theta_t]$ yields two smooth sections $\alpha, \beta$ of the vector bundle $R^2\pi_*\R_\cX \otimes C^{\infty}_{\bD}$ over $\bD$ such that 
\begin{equation}
\label{ab}
[\theta_t]=\alpha_t+\beta_t
\end{equation}
for any $t\in \bD$, cf \cite[Lemma~6.11]{BGL}. 

Since $H^1(Y_t, \R)=H^1(Z_t, \R)=0$ for $t\in \bD^*$, one can see easily that there exist Kähler classes $\widetilde \alpha_t \in H^2(Y_t, \R)$ and $\widetilde \beta_t \in H^2(Z_t,\R)$ such that $\alpha_t=\pi_1^*\widetilde \alpha_t$ and $\beta_t=\pi_2^*\widetilde \beta_t$. In particular, for any $t\in \bD^*$ the singular Ricci flat Kähler metric $\omega_t\in [\theta_t]$ can be split as 
\begin{equation}
\label{split KE}
\omega_t= \pi_1^*\omega_{Y_t}+\pi_2^*\omega_{Z_t}
\end{equation}
where $\omega_{Y_t}\in [\widetilde \alpha_t]$ (resp. $\omega_{Z_t}\in [\widetilde \beta_t]$) is the singular Ricci flat Kähler metric in the given Kähler class. 

Finally, since $\alpha_t, \beta_t \in H^1(X_t, \PH_{X_t})$ for $t\neq 0$, we have 
\begin{equation}
\label{potentiels}
\alpha_0, \beta_0 \in H^1(X_0, \PH_{X_0})
\end{equation}
since the vector spaces $H^1(X_t, \PH_{X_t})\simeq \mathrm{ker}\big(H^2(X_t, \R)\to H^2(X_t, \cO_{X_t})\big)$ glue to a vector bundle on $\bD$, cf \cite[Lemma~6.12]{BGL}. 

\subsubsection{Splitting of the metric}
For any $t\in \bD$, the splitting $T_{X_t} = A_t\oplus B_t$ derived from \eqref{tangent split} is, over $X_t^{\rm reg}$, a splitting into orthogonal, parallel subbundles with respect to $\omega_t$ thanks to Theorem~\ref{thm poly}. In particular, one can decompose over $X_t^{\rm reg}$
\[\omega_t=\omega_{A_t}+\omega_{B_t}\]
where $\omega_{A_t}$ is a smooth, closed semipositive $(1,1)$-form such that $\ker(\omega_{A_t})=B_t$ and $\omega_{A_t}|_{A_t}$ is positive definite. When $t\neq 0$, the splitting above coincides with \eqref{split KE}. In particular, $\omega_{A_t}$ extends to a positive current with local (bounded) potentials. When $t=0$, that latter remains valid. 
\begin{lem}
\label{lem pot}
The smooth $(1,1)$-forms $\omega_{A_0}$ and $\omega_{B_0}$ on $X_{\rm reg}$ extend to positive currents with bounded local potentials. 
\end{lem}
This means that given any $x\in X$, there exists a neighborhood $U$ of $x$ and a bounded psh function $u$ on $U$ such that $\omega_{A_0}|_{U_{\rm reg}}=dd^c u|_{U_{\rm reg}}$ (same for $\omega_{B_0}$).

\begin{proof}[Sketch of proof]
We give the main ideas and refer to \cite[Proposition~6.13]{BGL} for details. In particular, we admit that we have smooth convergence $\omega_{A_t}\to \omega_{A_0}$ on $\Xr$ when $t\to 0$ under local trivializations of $\pi$.
Let $p:\widetilde \cX\to\cX$ be a simultaneous resolution of singularities and set $\widetilde \pi := \pi \circ p$. It is not difficult to see that one can extract weak sequential limits $\widetilde T$ on $\widetilde X$ of $p_t^*\omega_{A_t}$ when $t\to 0$. Each such limit is a closed, positive $(1,1)$-current on $\widetilde X$ which coincides with $p_0^*\omega_{A_0}$ away from the exceptional locus of $\wX\to X$. Since we have an injection of vector bundles
 \[p^*R^2\pi_*\R_{\cX}\otimes C^{\infty}_\bD \hookrightarrow R^2\widetilde \pi_* \R_{\widetilde \cX} \otimes C^{\infty}_\bD\]
 one infers from \eqref{potentiels} that the cohomology class $[\widetilde T]\in H^2(\wX, \R)$ actually satisfies $[\widetilde T]\in p_0^*H^1(X, \PH_X)$. Since $\widetilde T\ge 0$, this implies that $\widetilde T=p_0^*T$ for some positive current $T\ge0$ on $X$ with local potentials. In particular, $T$ provides an extension of $\omega_{A_0}$ to $X$ with local potentials. Such an extension is unique. Indeed, if $U$ is a neighborhood of a given point $x\in X$ and if $u$ is a pluriharmonic function on $U_{\rm reg}$, then $u$ extends to a pluriharmonic function on $U$ by Grauert's theorem. 
 
Now, let $U$ be a neighborhood of a given point and let $\phi$ (resp. $u,v$) be a local potential for $\omega$ (resp. $\omega_{A_0},\omega_{B_0}$) on $U$. We can assume that $\phi=u+v$. Since $\phi$ is bounded and $u$ is bounded above, $v=\phi-u$ is bounded below, hence bounded. Same goes for $u$, and the proof is complete.
\end{proof}

In summary, we have proved: 

\begin{prop}
\label{prop KE splitting}
Let $X$ as in Corollary~\ref{cor split away} such that $K_X\sim \cO_X$, and let $T_X=A_0\oplus B_0$ be the induced splitting of the tangent sheaf, cf \eqref{tangent split}. The singular Kähler-Ricci flat metric $\omega\in [\theta_0]$ can be decomposed as a sum 
\[\omega=\omega_{A_0}+\omega_{B_0}\]
of closed positive $(1,1)$-currents with bounded local potentials satisfying additionally the following identities over $\Xr$: 
\[\ker(\omega_{A_0})=B_0, \quad \mbox{ and} \quad  \ker(\omega_{B_0})=A_0.\]
\end{prop}

\subsection{Completion of the proof}
\label{sec completion}
Let $X$ as in Setup~\ref{setup}. Thanks to Proposition~\ref{torus cover}, one can assume that $\widetilde q(X)=0$. It $T_X$ is strongly stable, then one can conclude by Proposition~\ref{prop strongly stable}. Otherwise, we are in the situation of Corollary~\ref{cor split away} and it is enough to show that one can split $X\simeq Y\times Z$ maybe after a passing to a quasi-étale cover.  

As in the proof of the corollary, we introduce the irreducible component $D_Y\subset \cD(\cX/\bD)$ whose general element over $t\in \bD^*$ parametrizes $Y_t\times \{z_t\}$ with $z_t\in Z_t^{\rm reg}$. The component $D_Y$ is proper over $\bD$, and the evaluation map $e_Y:\cF_Y\to \cX$ of the associated universal family $\cF_Y\subset D_Y \times \cX$ is isomorphic over $\bD^*$. By properness, $e_Y$ is surjective. For dimensional reason, the hypersurface $X_0$ is not in the image of the exceptional locus of $e_Y$ hence it makes sense to consider its proper transform $(\cF_Y)_0^+$. We let $(D_Y)_0^+$ be the image of $(\cF_Y)_0^+$ under the projection map $\cF_Y\to D_Y$. So we have a diagram
 \begin{equation}
 \label{diag Y}
\begin{tikzcd}
 (\cF_Y)_0^+ \arrow[d] \arrow[r,"e"] & X \\
 (D_Y)_0^+ & 
 \end{tikzcd}
\end{equation}
and $e$ is surjective, generically one-to-one. Similarly, we define $(\cF_Z)_0^+$ and $(D_Z)_0^+$. Let $[V_0] \in  (D_Y)_0^+$. It is not hard to see that $V_0\cap \Xr$ is either empty or tangent to $A_0|_{\Xr}$. In the second case, $V_0$ is the closure of a leaf of the foliation $A_0$. In any case, we have the following result which is \cite[Lemma~7.1]{BGL}. Its proof is a bit technical and will be omitted here.  

\begin{lem}
\label{irr red}
Let $[V_0] \in  (D_Y)_0^+$. Then $V_0$ is irreducible and generically reduced. 
\end{lem}

As a consequence of the above lemma and using the singular Ricci-flat metric, one can show that any pair of leaves of the two foliations have finite intersection. 
\begin{lem}
\label{inter finie}
Let $[V_0] \in  (D_Y)_0^+$ and let $[U_0]\in (D_Z)_0^+$. Then $U_0\cap V_0$ is a finite set. 
\end{lem}

\begin{proof}
Let $r:= \dim U_0\cap V_0$. Since $\omega$ has bounded potentials and $[\omega]$ is a Kähler class, it is enough to show that $\int_{U_0\cap V_0} \omega^r=0$. Since $\omega=\omega_{A_0}+\omega_{B_0}$ and each current in the decomposition admits bounded potentials, it will be enough to show that $\omega_{A_0}|_{U_0} \equiv 0$ and $\omega_{B_0}|_{V_0} \equiv 0$. By symmetry, it is enough to deal with $\omega_{A_0}|_{U_0}$.

If $U_0\cap \Xr\neq \emptyset$, then we see from Proposition~\ref{prop KE splitting} that $\omega_{A_0}|_{U_0}$ is supported on the proper analytic subset $U_0^{\rm sing}$. By the support theorem applied in a resolution of $U_0$, this contradicts the local boundedness of the potentials of the current. 

If $U_0\cap \Xr=\emptyset$, then we pick a general curve $C$ sending $0$ to $[U_0]\subset (D_Z)_0^+$ and get a pullback diagram
 \[
\begin{tikzcd}
\cF \arrow[d, "p"] \arrow[r,"e"] & X \\
C & 
 \end{tikzcd}
\]
For $t\in C$ general, we have $e^*\omega_{A_0}|_{F_t}=0$. Moreover, $p$ is smooth at a general point of $F_0=U_0$ by Lemma~\ref{irr red} above. Since $e^*\omega_{A_0}$ has bounded local potentials, we can appeal to the pluripotential theory lemma \cite[Lemma~7.3]{BGL} to infer that $e^*\omega_{A_0}$ vanishes in restriction to a Zariski open subset of $F_0$, hence everywhere thanks to the support theorem again. 
\end{proof}

Let $I:=(\cF_Y)_0^+\times_X (\cF_Z)_0^+$ be the universal intersection equipped with the evaluation map $I\to X$. There is a unique irreducible component $I^+$ of $I$ which dominate $X$. We have the following diagram 

 \[
\begin{tikzcd}
I^+ \arrow[d, "g"] \arrow[r,"f"] & X \\
(D_Y)_0^+\times  (D_Z)_0^+ & 
 \end{tikzcd}
\]

\begin{lem}
\label{wrap up}
The morphism $f$ is surjective and generically one-to-one. The morphism $g$ is finite, surjective and generically one-to-one. 
\end{lem}

\begin{proof}
Since the evaluation map $(\cF_Y)_0^+\to X$ in \eqref{diag Y} (and the same for $(\cF_Z)_0^+$) is generically one-to-one and surjective, so is $f$. Similarly, surjectivity of $(\cF_Y)_0^+\to (D_Y)_0^+$ (and the same for $(\cF_Z)_0^+$) implies that $g$ is surjective. The fibers of $g$ are finite thanks to Lemma~\ref{inter finie}.  

It remains to show that $g$ is generically one-to-one. Since $g$ is finite and $f^{-1}(X_{\rm sing})$ is a proper analytic subset of $I^+$, a general fiber of $g$ is made up of two leaves $V_0$ and $U_0$ of $A_0$ and $B_0$ such that $U_0\cap V_0 \subset \Xr$. We have to show that they only intersect in one point. If not, one could deform the leaves to the nearby variety $X_t\simeq Y_t\times Z_t$ while preserving at least two intersection points (cf \cite[Corollary~7.4]{BGL} for some more details), which is absurd. 
\end{proof}

We are now ready to conclude the proof of the main theorem. 
\begin{proof}[End of the proof of Theorem~\ref{thm BB klt}, Kähler case]
We borrow the notation set up at the beginning of Section~\ref{sec completion}. We need to show that $X$ splits accordingly to its tangent decomposition \eqref{tangent split}. By Lemma~\ref{wrap up}, the normalization of $I^+$ is isomorphic $Y\times Z$ where $Y$ (resp. $Z$) is the normalization of  $(D_Z)_0^+$ (resp. $(D_Y)_0^+$). Therefore we have a bimeromorphic, surjective morphism
\[f:Y\times Z\to X,\]
and the conclusion follows from the splitting result \cite[Lemma~7.5]{BGL} outlined in Remark~\ref{rem split}.
\end{proof}

\section{Further developments}
 In this last section, we would like to provide a few additional results as well as applications of the decomposition theorem and recent related results. 
 
\subsection{Fundamental groups}
\label{sec fg}
From a topological point of view, the major open question left is probably the following. 

\begin{conj}
\label{conj fg}
Let $X$ be an ICY or IHS variety. Then $\pi_1(\Xr)$ is finite. 
\end{conj}

By Theorem~\ref{alg approx}, it is enough to show the conjecture for projective varieties. Moreover, Theorem~\ref{thm BB klt} shows that Conjecture~\ref{conj fg} is equivalent to asking that given any compact normal variety $X$ with klt singularities, $c_1(X)=0$ and $\widetilde q(X)=0$, then $\pi_1(\Xr)$ is finite. By Proposition~\ref{torus cover}, it would imply that any compact normal variety $X$ with klt singularities satisfying $c_1(X)=0$ is such that $\pi_1(\Xr)$ is virtually abelian.  

Note that Conjecture~\ref{conj fg} is a consequence of the more general abelianity conjecture formulated by Campana. \\

{\it What is known.}
\begin{enumerate}[label=$(\roman*)$]
\item If $X$ is an IHS variety (resp. an ICY variety of even dimension) then $\pi_1(X)$ is trivial, cf \cite[Corollary~13.3]{GGK}. The main point is that $\chi(X, \cO_X)\neq 0$, allowing to use results of Campana.
\item If $X$ is as in Setup~\ref{setup} with $\widetilde q(X)=0$, then any linear representation of $\pi_1(\Xr)$ over any field is finite. Moreover, for each $r\in \N$, $\pi_1(\Xr)$ has only finitely many $r$-dimensional complex representations up to conjugation, cf Remark~\ref{rem flat} and \cite[Corollary~13.10]{GGK}.
\item If $X$ is an ICY threefold, then $\pi_1(\Xr)$ is finite, cf \cite[Theorem~6.1]{BF24} and \cite{TianXu17}.
\end{enumerate}
\subsection{Some applications}

\subsubsection{Primitive symplectic varieties}
\label{PSV}

In the literature, one often encounters the notion of primitive symplectic variety, cf \cite[Section~3]{BL22} and references therein. Recall that a symplectic variety (in the sense of Beauville) is a normal Kähler variety $X$ with a holomorphic symplectic closed form $\sigma \in H^0(\Xr, \Omega_{\Xr}^2)$ extending holomorphically to a holomorphic $2$-form on a resolution $Y\to X$. Moreover, $X$ is called primitive if $H^1(X,\cO_X)=0$ and $H^2(\Xr, \Omega_{\Xr}^2)=\C \sigma$. 

It is immediate to see that primitive symplectic varieties have canonical singularities. Moreover, IHS varieties are primitive symplectic. As a corollary of the decomposition theorem, one can see that these two notions are tightly connected.

\begin{lem}
Let $X$ be a primitive symplectic variety in the sense above. Then $X$ is either a torus quotient $X=T/G$ or it is of the form $X=Y^k/G$ where $Y$ is an IHS variety and $G\subset \mathfrak S_k$ acts transitively. 
\end{lem}

\begin{rem} A couple of comments. 

$\bullet$ If $X$ is smooth, then $X$ is already an IHS manifold by the main result of  \cite{Schwald}.

$\bullet$ There are plenty of such (singular) torus quotients. The simplest one is the singular Kummer surface, cf Example~\ref{sing kummer}, cf also \cite[Example~14.4]{GGK} for an example in dimension $4$ and \cite[Section~7]{Schwald} for higher dimensional examples provided with the help of GAP. 

$\bullet$ Any quotient $Y^k/G$ with $Y$ IHS and $G\subset \mathfrak S_k$ acting transitively is a primitive symplectic variety.
\end{rem}

\begin{proof}
Let $f:Y\to X$ be a quasi-étale Galois cover where $Y=T\times \prod_{i\in I} Y_i \times \prod_{j\in J} Z_j$ where $T$ is a torus, $Y_i$ are ICY varieties and $Z_j$ are IHS varieties. The existence of a such a cover $f$ is guaranteed by Theorem~\ref{thm BB klt}; the fact that $f$ can be chosen Galois follows e.g. from \cite[Lemma~2.8]{CGGN} since taking further covers does not change the splitting. Let $G$ be the Galois group of $f$. 

By the uniqueness property in Theorem~\ref{thm poly}, $G$ preserves $T$, $\prod_{i\in I} Y_i$ and  $\prod_{j\in J} Z_j$. By Bochner principle, $Y$ carries a unique $G$-invariant reflexive $2$-form $\sigma_Y$ and one can decompose 
\[\sigma_Y=\sigma_T+\sum_{i\in I} \sigma_i+\sum_{j\in J} \sigma_j\] where $\sigma_T\in H^0(T, \Omega_T^2)^G$, $\sigma_i\in H^0(Y_i,\Omega_{Y_i}^{[2]})=\{0\}$ and   $\sigma_j\in H^0(Y_i,\Omega_{Z_j}^{[2]})$.

Since $\sigma_Y$ is symplectic, one must have $I=\emptyset$. Since $\sigma_Y$ is unique and $G$ preserves $T$ and $\prod_{j\in J}Z_j$, one of the two factors must be trivial. If $J=\emptyset$, then $X$ is a torus quotient.

So we can assume that $Y=\prod_{j\in J} Z_j$. It follows from holonomy consideration (cf Section~\ref{sec rest hol}) that one can find a partition $J=\sqcup_{k\in K} J_k$ and a decomposition $G=\prod_{k\in K} G_k$ where $G_k\subset \mathfrak S_{J_k}$ acts (transitively) on $Z^k:=\prod_{j\in J_k} Z_j$ by permuting the factors. In particular, all $(Z_j)_{j\in J_K}$ are isomorphic. It also follows that one can decompose 
\[\sigma_Y=\sum_{k\in K} \sigma_k\] 
where $\sigma_k\in H^0(Z^k, \Omega_{Z^k}^{[2]})$ is $G_k$-invariant, hence $G$-invariant too. The uniqueness (up to scalar multiple) of $\sigma_Y$ implies that $|K|=1$ and the lemma follows.
\end{proof}

\subsubsection{Numerical charaterization of torus quotients}

Thanks to Yau's theorem, a compact Kähler manifold $X$ of dimension $n$ with $c_1(X)=0$ admits an étale cover $T\to X$ where $T$ is an $n$-dimensional complex torus if, and only if there exists a Kähler class $\alpha \in H^2(X,\R)$ such that $c_2(X)\cdot \alpha^{n-2}=0$.

This criterion was extended to projective varieties with klt singularities when $\alpha=c_1(H)$ is the class of an ample line bundle by Greb-Kebekus-Peternell \cite{GKP16} using slicing arguments when $X$ is smooth in codimension two. That restriction was later removed using the formalism of orbifold Chern classes by Lu-Taji \cite{LuTaji}. 

In the Kähler case such slicing arguments are not available. However, the decomposition theorem, i.e. Theorem~\ref{thm BB klt}, allows one (with extra work) to showing that ICY and IHS varieties have "positive second Chern class", which can be done separately taking advantage of the additional geometric properties. This was achieved in \cite{CGG} when $X$ is smooth in codimension two and in \cite{CGG2} for the general case even allowing a boundary. In summary, one has 

\begin{thm}[\cite{CGG,CGG2}]
Let $(X,\Delta)$ be a $n$-dimensional klt pair, where $X$ is a compact Kähler variety and $\Delta$ has standard coefficients, i.e. $\Delta=\sum_{i\in I}(1-\frac 1{m_i}) \Delta_i$ with $m_i\ge 2$ integers and $\Delta_i$ irreducible pairwise distinct. The following are equivalent: 
\begin{enumerate}[label=$(\roman*)$]
\item $c_1(K_X+\Delta)=0\in H^2(X,\R)$ and $c_2(X,\Delta)\cdot \alpha^{n-2}=0$ for some Kähler class $\alpha$. 
\item There exists a finite Galois cover $f:T\to X$ where $T$ is a complex torus and the branching divisor of $f$ coincides with $\Delta$. 
\end{enumerate}
\end{thm}

\subsubsection{Fujiki manifolds with trivial first Chern class}
Compact, non-K\"ahler manifolds with trivial first Chern class (or even trivial canonical bundle) show up naturally when performing small resolutions of singularities or small modifications starting from (possibly singular) compact K\"ahler varieties with trivial canonical bundle. There is abundant literature on the subject, and we refer to \cite[Section 3]{BCDG} and the references therein for a quick overview.

It makes sense to extend the notions of ICY and IHS manifolds (or varieties) given in Definition~\ref{defi icy ihs} to compact complex manifolds (or varieties) $X$ which are not necessarily K\"ahler. It was conjectured in \cite[Conjecture  3]{BCDG} that a compact complex manifold $X$ which is bimeromorphic to a K\"ahler manifold (sometimes referred to as a Fujiki manifold) and has trivial canonical bundle should be made up from irreducible K-trivial K\"ahler varieties (i.e., tori, Calabi–Yau or symplectic holomorphic) using small modifications, products and finite \'etale quotients. Moreover, $X$ should also satisfy a Beauville-Bogomolov type decomposition theorem. 

Based on Theorem~\ref{thm BB klt} and the existence of minimal models \cite{BCHM, HP16, DH24, DHP} the above conjecture was proved in \cite[Theorem~A]{BCDG} for Moishezon manifolds (i.e. compact complex manifolds bimeromorphic to a projective manifold) and Fujiki manifolds of dimension no greater than four. Since then the existence of minimal models for compact Kähler manifolds of numerical dimension zero has been obtained in full generality as a consequence of the (unconditional) extension of the results of \cite{BCHM} to the Kähler setting by \cite{HaconXie26}, see \cite[Theorem~6.3]{MZ26}. More precisely, it was proved by \cite{Druel11, Gongyo11} that the results of \cite{BCHM} imply the existence of a minimal model for any projective manifold with numerical Kodaira zero and their proof applies {\it mutatis mutandis} to the Kähler setting thanks to  \cite{HaconXie26}. 

Putting together the arguments in \cite{BCDG} and \cite[Theorem~6.3]{MZ26} yields a positive answer to \cite[Conjecture~3]{BCDG}.

\begin{thm}
Let $X$ be a compact Fujiki manifold such that $c_1(X)=0\in {\rm H}^2(X, \R)$. Then, there exists a finite \'etale cover $X'\to X$ and a decomposition
\[X'\simeq T\times \prod_{i\in I} Y_i \times \prod_{j\in J} Z_j\]
where $T$ is a compact complex torus, the $Y_i$'s are irreducible Calabi--Yau manifolds and the $Z_j$'s are irreducible holomorphic symplectic manifolds. 

Moreover, each factor $Y_i$ (respectively, $Z_j$) in the decomposition is bimeromorphic to a K\"ahler variety with terminal singularities which is irreducible Calabi--Yau (respectively, irreducible holomorphic symplectic).
\end{thm}

\subsection{Varieties with nef anticanonical bundle}

In the last few decades, the classification of projective or Kähler manifolds with nef anticanonical class has drawn a lot of attention, starting at least with seminal work by Demailly-Peternell-Schneider \cite{DPS94,DPS}. The following major result was first proved for projective manifolds by Cao-Höring \cite{CH19} relying on earlier progress by Cao \cite{CaoAlb}. The general Kähler case was later established by Matsumura-Wang-Wu-Zhang \cite{MWWZ25}. More precisely, they followed the general strategy of Cao-Höring and capitalized on the breakthrough \cite{Ou25} (see also \cite{CP25}) combined with the results in \cite{CH24}. 
\begin{thm}[\cite{CH19, MWWZ25}]
\label{CH19}
Let $X$ be a compact Kähler manifold such that $-K_X$ is nef. Then, there exists a locally trivial fibration
\[f:X\to Y\]
such that the fiber $F$ is rationally connected and $Y$ satisfies $c_1(Y)=0$. Moreover, the fibration becomes trivial after performing base change $\widetilde Y\to Y$ to the universal cover of $Y$. 
\end{thm}

Given the Beauville-Bogomolov decomposition theorem, Theorem~\ref{CH19} implies that if $X$ is compact Kähler with $-K_X$ nef, then the universal cover $\wX$ of $X$ is isomorphic to $F\times \C^r \times \prod_{i\in I} Y_i \times \prod_{j\in J} Z_j$ where $Y_i$ are ICY manifolds and $Z_j$ are IHS manifolds. 

The above structure theorem was extended to projective klt pairs $(X,\Delta)$ with $-(K_X+\Delta)$ nef by Matsumura-Wang \cite{MWang23} after important work by Campana-Cao-Matsumura \cite{CCM}. The result reads as follows:

\begin{thm}[\cite{MWang23}]
\label{MW}
Let $(X,\Delta)$ be a projective klt pair such that $-(K_X+\Delta)$ is nef. Up to replacing $X$ with a quasi-étale cover, there exists a fibration
\[f:X\to Y\]
which is locally trivial (for the pair structure) and satisfies that the fiber $F$ is rationally connected and $Y$ is klt with $c_1(Y)=0$. Moreover, the fibration becomes trivial after performing base change $\widetilde Y\to Y$ to the universal cover of $Y$. 

In particular, if $c_1(K_X+\Delta)=0$, then up to finite quasi-étale cover, one has a decomposition $(X,\Delta)\simeq (F, \Delta|_F) \times Y$. 
\end{thm}

If one knew that odd dimensional ICY varieties had finite fundamental group then the results in Section~\ref{sec fg} it would imply that up to performing a further quasi-étale cover,  the universal cover of $X$ would be isomorphic to $F\times \C^r \times \prod_{i\in I} Y_i \times \prod_{j\in J} Z_j$ where $Y_i$ are ICY varieties and $Z_j$ are IHS varieties.

\subsection{No Beauville-Bogomolov decomposition for log canonical singularities}
In the recent preprint \cite{BFPT}, the authors prove that Theorem~\ref{alb} fails if one relaxes the assumption on the singularities of $X$. More precisely, they exhibit the following phenomenon:

\begin{thm}
\label{lc}
Given any integer $n\ge 4$, there exists a projective variety $X$ of dimension $n$ with log canonical singularities and trivial canonical bundle such that the Albanese map $f:X\to \mathrm{Alb}(X)$ is surjective but not isotrivial.
\end{thm}

Note that the Albanese map is the one constructed by Serre \cite{Serre58} for any variety. The statement \cite[Theorem~1.1]{BFPT} is much stronger, but the weaker result above is enough e.g. to deduce the following, as observed later by N. M\"uller.

\begin{cor}
Let $X$ be as in Theorem~\ref{lc} and let $H$ be an ample Cartier divisor. Then $T_X$ is \emph{not} polystable with respect to $H$.
\end{cor}

\begin{proof}
Set $A:=\mathrm{Alb}(X)$. Although this is not needed for the proof, we will use the fact that $A$ is actually an elliptic curve in order to simplify the exposition.  Consider surjective morphism of sheaves $T_X\to f^*T_A$ induced by $f$. Since $T_A$ is trivial, we get a surjective morphism $\Phi:T_X\to \mathcal O_X$.  
If $T_X$ were polystable, one could write it as $T_X=\oplus_{i\in I} F_i$ with $F_i$ stable with zero slope. Given any index $i\in I$, composing the injection $F_i\to T_X$ and $\Phi$ would yield a morphism $F_i\to \mathcal O_X$ which would therefore either be the zero map or an isomorphism by stability of $F_i$.

In particular we would get an injective morphism $f^*T_A\to T_X$, hence a holomorphic vector field on $X$ whose flow would locally analytically on $A$ trivialize the family near a general point, a contradiction.
\end{proof}

\bibliographystyle{smfalpha}
\bibliography{biblio}

\providecommand{\bysame}{\leavevmode ---\ }
\providecommand{\og}{``}
\providecommand{\fg}{''}
\providecommand{\smfandname}{\&}
\providecommand{\smfedsname}{\'eds.}
\providecommand{\smfedname}{\'ed.}
\providecommand{\smfmastersthesisname}{M\'emoire}
\providecommand{\smfphdthesisname}{Th\`ese}
\begin{thebibliography}{MWWZ25}

\bibitem[AD13]{AD}
{\scshape C.~Araujo {\normalfont \smfandname} S.~Druel} -- {\og On {Fano}
  foliations\fg}, \emph{Adv. Math.} \textbf{238} (2013), p.~70--118 (English).

\bibitem[AV21]{AV19}
{\scshape E.~Amerik {\normalfont \smfandname} M.~Verbitsky} -- {\og Contraction
  centers in families of hyperk\"{a}hler manifolds\fg}, \emph{Selecta Math.
  (N.S.)} \textbf{27} (2021), no.~4, p.~Paper No. 60, 26.

\bibitem[BCDG25]{BCDG}
{\scshape I.~Biswas, J.~Cao, S.~Dumitrescu {\normalfont \smfandname}
  H.~Guenancia} -- {\og Geometry of {{\(K\)}}-trivial {Moishezon} manifolds:
  decomposition theorem and holomorphic geometric structures\fg}, \emph{Math.
  Ann.} \textbf{391} (2025), no.~2, p.~3181--3220 (English).

\bibitem[BCHM10]{BCHM}
{\scshape C.~Birkar, P.~Cascini, C.~Hacon {\normalfont \smfandname}
  J.~McKernan} -- {\og {Existence of minimal models for varieties of log
  general type}\fg}, \emph{J. Amer. Math. Soc.} \textbf{23} (2010),
  p.~405--468.

\bibitem[Bea83]{Bea83}
{\scshape A.~Beauville} -- {\og Vari\'{e}t\'{e}s {K}\"{a}hleriennes dont la
  premi\`{e}re classe de {C}hern est nulle\fg}, \emph{J. Differential Geom.}
  \textbf{18} (1983), no.~4, p.~755--782 (1984).

\bibitem[BEG13]{BEG}
{\scshape S.~Boucksom, P.~Eyssidieux {\normalfont \smfandname} V.~Guedj} --
  {\og Introduction\fg}, in \emph{An introduction to the {K}\"ahler-{R}icci
  flow}, Lecture Notes in Math., vol. 2086, Springer, Cham, 2013, p.~1--6.

\bibitem[Bes87]{Besse}
{\scshape A.~L. Besse} -- \emph{Einstein manifolds}, Ergebnisse der Mathematik
  und ihrer Grenzgebiete (3) [Results in Mathematics and Related Areas (3)],
  vol.~10, Springer-Verlag, Berlin, 1987,
  \href{http://dx.doi.org/10.1007/978-3-540-74311-8}{DOI:10.1007/978-3-540-74311-8}.

\bibitem[BF24]{BF24}
{\scshape L.~Braun {\normalfont \smfandname} F.~Figueroa} -- {\og {Fundamental
  groups, coregularity, and low dimensional klt Calabi-Yau pairs}\fg}, Preprint
  \href{http://arxiv.org/abs/2401.01315}{arXiv:2401.01315}, 2024.

\bibitem[BFPT24]{BFPT}
{\scshape F.~Bernasconi, S.~Filipazzi, Z.~Patakfalvi {\normalfont \smfandname}
  N.~Tsakanikas} -- {\og {A counterexample to the log canonical
  Beauville--Bogomolov decomposition}\fg}, Preprint
  \href{https://arxiv.org/abs/2407.17260}{arXiv:2407.17260}, 2024.

\bibitem[BGL22]{BGL}
{\scshape B.~Bakker, H.~Guenancia {\normalfont \smfandname} C.~Lehn} -- {\og
  Algebraic approximation and the decomposition theorem for {K}\"{a}hler
  {C}alabi-{Y}au varieties\fg}, \emph{Invent. Math.} \textbf{228} (2022),
  no.~3, p.~1255--1308.

\bibitem[BGMM25]{BGMM25}
{\scshape V.~Bertini, A.~Grossi, M.~Mauri {\normalfont \smfandname} E.~Mazzon}
  -- {\og Terminalizations of quotients of compact hyperk{\"a}hler manifolds by
  induced symplectic automorphisms\fg}, \emph{{\'E}pijournal de G{\'e}om.
  Alg{\'e}br., EPIGA} \textbf{9} (2025), p.~53 (English), Id/No 14.

\bibitem[BH99]{BH99}
{\scshape M.~R. Bridson {\normalfont \smfandname} A.~Haefliger} -- \emph{Metric
  spaces of non-positive curvature}, Grundlehren der mathematischen
  Wissenschaften [Fundamental Principles of Mathematical Sciences], vol. 319,
  Springer-Verlag, Berlin, 1999.

\bibitem[BL22]{BL22}
{\scshape B.~Bakker {\normalfont \smfandname} C.~Lehn} -- {\og The global
  moduli theory of symplectic varieties\fg}, \emph{J. Reine Angew. Math.}
  \textbf{790} (2022), p.~223--265 (English).

\bibitem[BM16]{bogomolov_mcquillan01}
{\scshape F.~Bogomolov {\normalfont \smfandname} M.~McQuillan} -- {\og Rational
  curves on foliated varieties\fg}, in \emph{Foliation theory in algebraic
  geometry}, Simons Symp., Springer, Cham, 2016, p.~21--51.

\bibitem[Bog74]{Bog74}
{\scshape F.~A. Bogomolov} -- {\og The decomposition of {K}\"ahler manifolds
  with a trivial canonical class\fg}, \emph{Mat. Sb. (N.S.)} \textbf{93(135)}
  (1974), p.~573--575, 630.

\bibitem[Bos01]{Bost01}
{\scshape J.-B. Bost} -- {\og Algebraic leaves of algebraic foliations over
  number fields.\fg}, \emph{Publ. Math., Inst. Hautes {\'E}tud. Sci.}
  \textbf{93} (2001), p.~161--221 (English).

\bibitem[Bra21]{Braun21}
{\scshape L.~Braun} -- {\og The local fundamental group of a {Kawamata} log
  terminal singularity is finite\fg}, \emph{Invent. Math.} \textbf{226} (2021),
  no.~3, p.~845--896 (English).

\bibitem[Bru15]{Brunella}
{\scshape M.~Brunella} -- \emph{Birational geometry of foliations}, reprint of
  the 2000 edition with new results \smfedname, IMPA Monogr., vol.~1, Cham:
  Springer, 2015 (English).

\bibitem[BT82]{BottTu}
{\scshape R.~Bott {\normalfont \smfandname} L.~W. Tu} -- \emph{Differential
  forms in algebraic topology}, Grad. Texts Math., vol.~82, Springer, Cham,
  1982 (English).

\bibitem[Cam04]{Campana04}
{\scshape F.~Campana} -- {\og Orbifolds, special varieties and classification
  theory\fg}, \emph{Ann. Inst. Fourier (Grenoble)} \textbf{54} (2004), no.~3,
  p.~499--630.

\bibitem[Cam21]{CampBB}
\bysame , {\og The {Bogomolov}-{Beauville}-{Yau} decomposition for klt
  projective varieties with trivial first {Chern} class -- without tears\fg},
  \emph{Bull. Soc. Math. Fr.} \textbf{149} (2021), no.~1, p.~1--13 (English).

\bibitem[Cao19]{CaoAlb}
{\scshape J.~Cao} -- {\og Albanese maps of projective manifolds with nef
  anticanonical bundles\fg}, \emph{Ann. Sci. {\'E}c. Norm. Sup{\'e}r. (4)}
  \textbf{52} (2019), no.~5, p.~1137--1154 (English).

\bibitem[CCE15]{CCE15}
{\scshape F.~Campana, B.~Claudon {\normalfont \smfandname} P.~Eyssidieux} --
  {\og Repr\'{e}sentations lin\'{e}aires des groupes k\"{a}hl\'{e}riens:
  factorisations et conjecture de {S}hafarevich lin\'{e}aire\fg}, \emph{Compos.
  Math.} \textbf{151} (2015), no.~2, p.~351--376.

\bibitem[CCM21]{CCM}
{\scshape F.~Campana, J.~Cao {\normalfont \smfandname} S.-i. Matsumura} -- {\og
  Projective klt pairs with nef anti-canonical divisor\fg}, \emph{Algebr.
  Geom.} \textbf{8} (2021), no.~4, p.~430--464 (English).

\bibitem[CCP21]{CCP21}
{\scshape F.~Campana, J.~Cao {\normalfont \smfandname} M.~Păun} -- {\og
  Subharmonicity of direct images and applications\fg}, Preprint
  \href{https://arxiv.org/abs/1906.11317}{arXiv:1906.11317}, 2021.

\bibitem[CG71]{CG71}
{\scshape J.~Cheeger {\normalfont \smfandname} D.~Gromoll} -- {\og The
  splitting theorem for manifolds of nonnegative {Ricci} curvature\fg},
  \emph{J. Differ. Geom.} \textbf{6} (1971), p.~119--128 (English).

\bibitem[CGG22]{CGG}
{\scshape B.~Claudon, P.~Graf {\normalfont \smfandname} H.~Guenancia} -- {\og
  Numerical characterization of complex torus quotients\fg}, \emph{Comment.
  Math. Helv.} \textbf{97} (2022), no.~4, p.~769--799 (English).

\bibitem[CGG24]{CGG2}
\bysame , {\og Equality in the {Miyaoka}-{Yau} inequality and uniformization of
  non-positively curved klt pairs\fg}, \emph{C. R., Math., Acad. Sci. Paris}
  \textbf{362} (2024), no.~S1, p.~55--81 (English).

\bibitem[CGGN22]{CGGN}
{\scshape B.~Claudon, P.~Graf, H.~Guenancia {\normalfont \smfandname}
  P.~Naumann} -- {\og K\"{a}hler spaces with zero first {C}hern class:
  {B}ochner principle, {A}lbanese map and fundamental groups\fg}, \emph{J.
  Reine Angew. Math.} \textbf{786} (2022), p.~245--275.

\bibitem[CGP23]{JHM2}
{\scshape J.~Cao, H.~Guenancia {\normalfont \smfandname} M.~{P\u{a}un}} -- {\og
  Variation of singular {K{\"a}hler}-{Einstein} metrics: {Kodaira} dimension
  zero (with an appendix by {Valentino} {Tosatti})\fg}, \emph{J. Eur. Math.
  Soc. (JEMS)} \textbf{25} (2023), no.~2, p.~633--679 (English).

\bibitem[CH19]{CH19}
{\scshape J.~Cao {\normalfont \smfandname} A.~H{\"o}ring} -- {\og A
  decomposition theorem for projective manifolds with nef anticanonical
  bundle\fg}, \emph{J. Algebr. Geom.} \textbf{28} (2019), no.~3, p.~567--597
  (English).

\bibitem[CH24]{CH24}
{\scshape B.~Claudon {\normalfont \smfandname} A.~H{ö}ring} -- {\og
  Projectivity criteria for {Kähler} morphisms\fg}, Preprint
  \href{http://arxiv.org/abs/2404.13927}{arXiv:2404.13927}, 2024.

\bibitem[CP19]{CP19}
{\scshape F.~Campana {\normalfont \smfandname} M.~P\u{a}un} -- {\og Foliations
  with positive slopes and birational stability of orbifold cotangent
  bundles\fg}, \emph{Publ. Math. Inst. Hautes \'{E}tudes Sci.} \textbf{129}
  (2019), p.~1--49.

\bibitem[CP25]{CP25}
{\scshape J.~Cao {\normalfont \smfandname} M.~P{ă}un} -- {\og {Remarks on
  relative canonical bundles and algebraicity criteria for foliations in
  Kähler context}\fg}, Preprint
  \href{https://arxiv.org/abs/2502.02183}{arXiv:2502.02183}, 2025.

\bibitem[Dem85]{Dem85}
{\scshape J.-P. Demailly} -- {\og Mesures de {M}onge-{A}mp{\`e}re et
  caract{\'e}risation g{\'e}om{\'e}trique des vari{\'e}t{\'e}s alg{\'e}briques
  affines\fg}, \emph{M{\'e}m. Soc. Math. France (N.S.)} (1985), no.~19, p.~124.

\bibitem[DGP24]{DGP}
{\scshape S.~Druel, H.~Guenancia {\normalfont \smfandname} M.~P{\u{a}}un} --
  {\og A decomposition theorem for {{\(\mathbb{Q}\)}}-{Fano}
  {K{\"a}hler}-{Einstein} varieties\fg}, \emph{C. R., Math., Acad. Sci. Paris}
  \textbf{362} (2024), no.~S1, p.~93--118 (English).

\bibitem[DH24]{DH24}
{\scshape O.~Das {\normalfont \smfandname} C.~Hacon} -- {\og {On the Minimal
  Model Program for K\"ahler 3-folds}\fg}, Preprint
  \href{https://arxiv.org/abs/2306.11708}{arXiv:2306.11708}, 2024.

\bibitem[DHP24]{DHP}
{\scshape O.~Das, C.~Hacon {\normalfont \smfandname} M.~P{\u{a}}un} -- {\og On
  the 4-dimensional minimal model program for {K{\"a}hler} varieties\fg},
  \emph{Adv. Math.} \textbf{443} (2024), p.~68 (English), Id/No 109615.

\bibitem[DPS]{DPS}
{\scshape J.-P. Demailly, T.~Peternell {\normalfont \smfandname} M.~Schneider}
  -- {\og {K{\"a}hler manifolds with numerically effective Ricci class}\fg},
  \emph{Comp. Math.}

\bibitem[DPS94]{DPS94}
\bysame , {\og Compact complex manifolds with numerically effective tangent
  bundles\fg}, \emph{J. Algebr. Geom.} \textbf{3} (1994), no.~2, p.~295--345
  (English).

\bibitem[Dru11]{Druel11}
{\scshape S.~Druel} -- {\og Some remarks on the divisorial {Zariski}
  decomposition on varieties whose first {Chern} class is zero\fg}, \emph{Math.
  Z.} \textbf{267} (2011), no.~1-2, p.~413--423 (French).

\bibitem[Dru14]{DruelZL}
\bysame , {\og The {Zariski}-{Lipman} conjecture for log canonical spaces\fg},
  \emph{Bull. Lond. Math. Soc.} \textbf{46} (2014), no.~4, p.~827--835
  (English).

\bibitem[Dru18]{Dru16}
\bysame , {\og A decomposition theorem for singular spaces with trivial
  canonical class of dimension at most five\fg}, \emph{Invent. Math.}
  \textbf{211} (2018), no.~1, p.~245--296.

\bibitem[EGZ09]{EGZ}
{\scshape P.~Eyssidieux, V.~Guedj {\normalfont \smfandname} A.~Zeriahi} -- {\og
  {Singular K{\"a}hler-Einstein metrics}\fg}, \emph{{J. Amer. Math. Soc.}}
  \textbf{22} (2009), p.~607--639.

\bibitem[Eno88]{Enoki}
{\scshape I.~Enoki} -- {\og Stability and negativity for tangent sheaves of
  minimal {K}{\"a}hler spaces\fg}, in \emph{Geometry and analysis on manifolds
  ({K}atata/{K}yoto, 1987)}, Lecture Notes in Math., vol. 1339, Springer,
  Berlin, 1988, p.~118--126.

\bibitem[Fuj78]{Fuj78}
{\scshape A.~Fujiki} -- {\og On automorphism groups of compact {K{\"a}hler}
  manifolds\fg}, \emph{Invent. Math.} \textbf{44} (1978), p.~225--258
  (English).

\bibitem[Fuj22]{Fujino22}
{\scshape O.~Fujino} -- {\og {Minimal model program for projective morphisms
  between complex analytic spaces}\fg}, Preprint
  \href{https://arxiv.org/abs/2201.11315}{arXiv:2201.11315}, 2022.

\bibitem[Fuj79]{Fuj_Douady}
{\scshape A.~Fujiki} -- {\og Closedness of the {D}ouady spaces of compact
  {K}\"{a}hler spaces\fg}, \emph{Publ. Res. Inst. Math. Sci.} \textbf{14}
  (1978/79), no.~1, p.~1--52.

\bibitem[GGK19]{GGK}
{\scshape D.~Greb, H.~Guenancia {\normalfont \smfandname} S.~Kebekus} -- {\og
  Klt varieties with trivial canonical class: holonomy, differential forms, and
  fundamental groups\fg}, \emph{Geom. Topol.} \textbf{23} (2019),
  p.~2051--2124.

\bibitem[GKK10]{GKK}
{\scshape D.~Greb, S.~Kebekus {\normalfont \smfandname} S.~J. Kov{\'a}cs} --
  {\og Extension theorems for differential forms and {B}ogomolov-{S}ommese
  vanishing on log canonical varieties\fg}, \emph{Compos. Math.} \textbf{146}
  (2010), no.~1, p.~193--219.

\bibitem[GKKP11]{GKKP11}
{\scshape D.~Greb, S.~Kebekus, S.~J. Kov{\'a}cs {\normalfont \smfandname}
  T.~Peternell} -- {\og Differential forms on log canonical spaces\fg},
  \emph{Publ. Math., Inst. Hautes {\'E}tud. Sci.} \textbf{114} (2011),
  p.~87--169 (English).

\bibitem[GKP16a]{GKP16}
{\scshape D.~Greb, S.~Kebekus {\normalfont \smfandname} T.~Peternell} -- {\og
  \'{E}tale fundamental groups of {K}awamata log terminal spaces, flat sheaves,
  and quotients of abelian varieties\fg}, \emph{Duke Math. J.} \textbf{165}
  (2016), no.~10, p.~1965--2004.

\bibitem[GKP16b]{GKP}
\bysame , {\og Singular spaces with trivial canonical class\fg}, in
  \emph{Minimal Models and Extremal Rays, Kyoto, 2011}, Adv.~Stud.~Pure Math.,
  vol.~70, Mathematical Society of Japan, Tokyo, 2016, p.~67--113.

\bibitem[GM88]{GMbook}
{\scshape M.~Goresky {\normalfont \smfandname} R.~MacPherson} --
  \emph{Stratified {M}orse theory}, Ergebnisse der Mathematik und ihrer
  Grenzgebiete (3) [Results in Mathematics and Related Areas (3)], vol.~14,
  Springer-Verlag, Berlin, 1988.

\bibitem[Gon11]{Gongyo11}
{\scshape Y.~Gongyo} -- {\og On the minimal model theory for {DLT} pairs of
  numerical log {Kodaira} dimension zero\fg}, \emph{Math. Res. Lett.}
  \textbf{18} (2011), no.~5, p.~991--1000 (English).

\bibitem[GP24]{GP24}
{\scshape H.~Guenancia {\normalfont \smfandname} M.~P{\u{a}}un} -- {\og {
  Bogomolov-Gieseker inequality for log terminal Kähler threefolds}\fg},
  Preprint \href{https://arxiv.org/abs/2405.10003}{arXiv:2405.10003}, to appear
  in Comm. Pure Appl. Math., 2024.

\bibitem[GPP24]{GPP24}
{\scshape A.~Garbagnati, M.~Penegini {\normalfont \smfandname} A.~Perego} --
  {\og Singular symplectic surfaces\fg},
  \href{https://arxiv.org/abs/2407.21173}{arXiv:2407.21173}, 2024.

\bibitem[GPSS23]{GPSS23}
{\scshape B.~Guo, D.~Phong, J.~Song {\normalfont \smfandname} J.~Sturm} -- {\og
  {Sobolev inequalities on Kähler spaces}\fg}, Preprint
  \href{https://arxiv.org/abs/2311.00221}{arXiv:2311.00221}, 2023.

\bibitem[Gra18]{Graf18}
{\scshape P.~Graf} -- {\og Algebraic approximation of {K{\"a}hler} threefolds
  of {Kodaira} dimension zero\fg}, \emph{Math. Ann.} \textbf{371} (2018),
  no.~1-2, p.~487--516 (English).

\bibitem[Gue16]{GSS}
{\scshape H.~Guenancia} -- {\og {Semistability of the tangent sheaf of singular
  varieties}\fg}, \emph{Algebraic Geometry} \textbf{3} (2016), no.~5,
  p.~508--542.

\bibitem[HP16]{HP16}
{\scshape A.~H{\"o}ring {\normalfont \smfandname} T.~Peternell} -- {\og Minimal
  models for {K{\"a}hler} threefolds\fg}, \emph{Invent. Math.} \textbf{203}
  (2016), no.~1, p.~217--264 (English).

\bibitem[HP19]{HP}
\bysame , {\og Algebraic integrability of foliations with numerically trivial
  canonical bundle\fg}, \emph{Invent. Math.} \textbf{216} (2019), no.~2,
  p.~395--419.

\bibitem[HS17]{HS}
{\scshape H.-J. Hein {\normalfont \smfandname} S.~Sun} -- {\og Calabi-{Y}au
  manifolds with isolated conical singularities\fg}, \emph{Publ. Math. Inst.
  Hautes \'Etudes Sci.} \textbf{126} (2017), p.~73--130.

\bibitem[HX26]{HaconXie26}
{\scshape C.~Hacon {\normalfont \smfandname} L.~Xie} -- {\og On the
  {K{\"a}hler} {MMP} and the transcendental base-point-free theorem\fg},
  Preprint \href{https://arxiv.org/2607.24986}{arXiv:2607.24986}, 2026.

\bibitem[Joy00]{Joy00}
{\scshape D.~D. Joyce} -- \emph{Compact manifolds with special holonomy},
  Oxford Mathematical Monographs, Oxford University Press, Oxford, 2000.

\bibitem[Joy07]{Joyce}
\bysame , \emph{Riemannian holonomy groups and calibrated geometry}, Oxf. Grad.
  Texts Math., vol.~12, Oxford: Oxford University Press, 2007 (English).

\bibitem[Kaw85]{Kawa85}
{\scshape Y.~Kawamata} -- {\og Minimal models and the {K}odaira dimension of
  algebraic fiber spaces\fg}, \emph{J. Reine Angew. Math.} \textbf{363} (1985),
  p.~1--46.

\bibitem[KL09]{kollar_larsen}
{\scshape J.~Koll\'{a}r {\normalfont \smfandname} M.~Larsen} -- {\og Quotients
  of {C}alabi-{Y}au varieties\fg}, in \emph{Algebra, arithmetic, and geometry:
  in honor of {Y}u. {I}. {M}anin. {V}ol. {II}}, Progr. Math., vol. 270,
  Birkh\"{a}user Boston, Boston, MA, 2009, p.~179--211.

\bibitem[KM98]{KM}
{\scshape J.~Koll{\'a}r {\normalfont \smfandname} S.~Mori} -- \emph{Birational
  geometry of algebraic varieties}, Cambridge Tracts in Mathematics, vol. 134,
  Cambridge University Press, Cambridge, 1998, With the collaboration of C. H.
  Clemens and A. Corti, Translated from the 1998 Japanese original.

\bibitem[KN96]{KN}
{\scshape S.~Kobayashi {\normalfont \smfandname} K.~Nomizu} --
  \emph{Foundations of differential geometry. {V}ol. {I}}, Wiley Classics
  Library, John Wiley \& Sons, Inc., New York, 1996, Reprint of the 1963
  original, A Wiley-Interscience Publication.

\bibitem[Kol97]{Kollar97}
{\scshape J.~Koll{\'a}r} -- {\og Singularities of pairs\fg}, in \emph{Algebraic
  geometry---{S}anta {C}ruz 1995}, Proc. Sympos. Pure Math., vol.~62, Amer.
  Math. Soc., Providence, RI, 1997, p.~221--287.

\bibitem[Kol13]{KollarMMP}
\bysame , \emph{Singularities of the minimal model program. {With} the
  collaboration of {S{\'a}ndor} {Kov{\'a}cs}}, Camb. Tracts Math., vol. 200,
  Cambridge: Cambridge University Press, 2013 (English).

\bibitem[KS21]{KS}
{\scshape S.~Kebekus {\normalfont \smfandname} C.~Schnell} -- {\og Extending
  holomorphic forms from the regular locus of a complex space to a resolution
  of singularities\fg}, \emph{J. Am. Math. Soc.} \textbf{34} (2021), no.~2,
  p.~315--368 (English).

\bibitem[LT18]{LuTaji}
{\scshape S.~Lu {\normalfont \smfandname} B.~Taji} -- {\og A characterization
  of finite quotients of abelian varieties\fg}, \emph{Int. Math. Res. Not.}
  \textbf{2018} (2018), no.~1, p.~292--319 (English).

\bibitem[LT19]{LiTian19}
{\scshape C.~Li {\normalfont \smfandname} G.~Tian} -- {\og Orbifold regularity
  of weak {K}\"{a}hler-{E}instein metrics\fg}, in \emph{Advances in complex
  geometry}, Contemp. Math., vol. 735, Amer. Math. Soc., [Providence], RI,
  [2019] \copyright 2019, p.~169--178.

\bibitem[MW23]{MWang23}
{\scshape S.-i. Matsumura {\normalfont \smfandname} J.~Wang} -- {\og Structure
  theorem for projective klt pairs with nef anti-canonical divisor\fg},
  \href{https://arxiv.org/abs/2105.14308}{arXiv:2105.14308}, to appear in JEMS,
  2023.

\bibitem[MWWZ25]{MWWZ25}
{\scshape S.-i. Matsumura, J.~Wang, X.~Wu {\normalfont \smfandname} Q.~Zhang}
  -- {\og Compact {Kähler} manifolds with nef anti-canonical bundle\fg},
  \href{https://arxiv.org/abs/2506.23218}{arXiv:2506.23218}, 2025.

\bibitem[MZ26]{MZ26}
{\scshape S.~Matsumura {\normalfont \smfandname} G.~Zhong} -- {\og Positivity
  of compact {K{\"a}hler} varieties admitting an int-amplified
  endomorphism\fg}, 2026, Preprint
  \href{https://arxiv.org/2609.09869}{arXiv:2609.09869}.

\bibitem[Nak04]{Nak}
{\scshape N.~Nakayama} -- \emph{Zariski-decomposition and abundance}, MSJ
  Memoirs, vol.~14, Mathematical Society of Japan, Tokyo, 2004.

\bibitem[Ou25]{Ou25}
{\scshape W.~Ou} -- {\og {A characterization of uniruled compact {Kähler}
  manifolds}\fg}, Preprint
  \href{https://arxiv.org/abs/2501.18088}{arXiv:2501.18088}, 2025.

\bibitem[P{\u{a}}u98]{Paun98}
{\scshape M.~P{\u{a}}un} -- {\og On the numerical effectiveness of inverse
  images of line bundles\fg}, \emph{Math. Ann.} \textbf{310} (1998), no.~3,
  p.~411--421 (French).

\bibitem[P{\u{a}}u08]{Paun}
\bysame , {\og {Regularity properties of the degenerate Monge-Amp{\`e}re
  equations on compact K{\"a}hler manifolds.}\fg}, \emph{Chin. Ann. Math., Ser.
  B} \textbf{29} (2008), no.~6, p.~623--630.

\bibitem[Pet94]{Pet94}
{\scshape T.~Peternell} -- {\og Minimal varieties with trivial canonical
  classes. {I}\fg}, \emph{Math. Z.} \textbf{217} (1994), no.~3, p.~377--405
  (English).

\bibitem[PR18]{PR18}
{\scshape A.~Perego {\normalfont \smfandname} A.~Rapagnetta} -- {\og The moduli
  spaces of semistable sheaves on {K3} surfaces are irreducible symplectic
  varieties\fg}, Preprint
  \href{http://arxiv.org/abs/1802.01182}{arXiv:1802.01182}, to appear in
  Algebraic Geometry, 2018.

\bibitem[PT13]{PT13}
{\scshape J.~V. Pereira {\normalfont \smfandname} F.~Touzet} -- {\og Foliations
  with vanishing {Chern} classes\fg}, \emph{Bull. Braz. Math. Soc. (N.S.)}
  \textbf{44} (2013), no.~4, p.~731--754 (English).

\bibitem[Sac23]{Sacca23}
{\scshape G.~Sacc{à}} -- {\og Singular symplectic surfaces\fg},
  \href{https://arxiv.org/abs/2304.02609}{arXiv:2304.02609}, to appear in
  "Current developments in Hodge theory. Proceedings of Hodge theory at IMSA'',
  Simons Symposia Series, 2023.

\bibitem[Sch71]{Sch71}
{\scshape M.~Schlessinger} -- {\og Rigidity of quotient singularities\fg},
  \emph{Invent. Math.} \textbf{14} (1971), p.~17--26.

\bibitem[Sch22]{Schwald}
{\scshape M.~Schwald} -- {\og On the definition of irreducible holomorphic
  symplectic manifolds and their singular analogs\fg}, \emph{Int. Math. Res.
  Not.} \textbf{2022} (2022), no.~15, p.~11864--11877 (English).

\bibitem[Ser60]{Serre58}
{\scshape J.-P. Serre} -- {\og Universal morphisms and {Albanese}
  varieties\fg}, Vari{\'e}t{\'e}s de {Picard}. {S{\'e}m}. {C}. {Chevalley} 3
  (1958/59), {No}. 10, 22 p. (1960)., 1960.

\bibitem[Tak03]{Takayama2003}
{\scshape S.~Takayama} -- {\og Local simple connectedness of resolutions of
  log-terminal singularities\fg}, \emph{Internat. J. Math.} \textbf{14} (2003),
  no.~8, p.~825--836.

\bibitem[TX17]{TianXu17}
{\scshape Z.~Tian {\normalfont \smfandname} C.~Xu} -- {\og Finiteness of
  fundamental groups\fg}, \emph{Compos. Math.} \textbf{153} (2017), no.~2,
  p.~257--273 (English).

\bibitem[Vie83]{Vieh83}
{\scshape E.~Viehweg} -- {\og Weak positivity and the additivity of the
  {K}odaira dimension for certain fibre spaces\fg}, in \emph{Algebraic
  varieties and analytic varieties ({T}okyo, 1981)}, Adv. Stud. Pure Math.,
  vol.~1, North-Holland, Amsterdam, 1983, p.~329--353.

\bibitem[Voi02]{VoisinI}
{\scshape C.~Voisin} -- \emph{Hodge theory and complex algebraic geometry. {I}.
  {Translated} from the {French} by {Leila} {Schneps}}, Camb. Stud. Adv. Math.,
  vol.~76, Cambridge: Cambridge University Press, 2002 (English).

\bibitem[Yau78]{Yau78}
{\scshape S.-T. Yau} -- {\og On the {R}icci curvature of a compact {K}{\"a}hler
  manifold and the complex {M}onge-{A}mp{\`e}re equation. {I}\fg}, \emph{Comm.
  Pure Appl. Math.} \textbf{31} (1978), no.~3, p.~339--411.

\end{thebibliography}

\end{document}